\documentclass[12pt]{amsart}
\usepackage[dvips]{graphicx}
\usepackage{amsmath,graphics,xypic}
\usepackage{amsfonts,amssymb}
\usepackage[all]{xy}
\theoremstyle{plain}
\newtheorem{theorem*}{Theorem}
\newtheorem*{lemma*}{Lemma}
\newtheorem{corollary*}{Corollary}
\newtheorem*{proposition*}{Proposition}
\newtheorem{conjecture*}{Conjecture}
\newtheorem{theorem}{Theorem}[section]
\newtheorem{lemma}[theorem]{Lemma}

\newtheorem{proposition}[theorem]{Proposition}

\theoremstyle{remark}

\newtheorem*{remark}{Remark}

\newtheorem*{claim}{Claim}

\theoremstyle{definition}

\def\sl{\mbox{SL}}
\def\gl{\mbox{GL}}

\def\sym{\mbox{Sym}}
\def\ss{\mathfrak{s}}
\def\kt{\KK\tpm}
\def\w{\omega}
\def\vect{\operatorname{vect}}
\def\eul{\operatorname{Eul}}
\def\spinc{\operatorname{Spin}^c}
\def\ort{\operatorname{Or}}
\def\gl{\mbox{GL}} \def\Q{\Bbb{Q}}
\def\HH{\mathcal{H}}
\def\BB{\mathcal{B}}
\def\KK{\Bbb{K}}  \def\Z{\Bbb{Z}} \def\R{\Bbb{R}} \def\C{\Bbb{C}}
  \def\l{\lambda} \def\ll{\langle} \def\rr{\rangle}
 \def\a{\alpha} \def\g{\gamma}  \def\bp{\begin{pmatrix}}
\def\sm{\setminus} \def\ep{\end{pmatrix}} \def\bn{\begin{enumerate}} 
   \def\en{\end{enumerate}}
\def\ba{\begin{array}} \def\ea{\end{array}}  
  \def\s{\sigma} \def\a{\alpha} \def\b{\beta} \def\ti{\tilde}\def\wti{\widetilde}
\def\id{\mbox{id}}  \def\im{\mbox{Im}} \def\sign{\mbox{sign}}
  \def\ker{\mbox{Ker}}
\def\ker{\mbox{Ker}}\def\be{\begin{equation}} \def\ee{\end{equation}} 
   
 \def\hom{\mbox{Hom}}  
 \def\aut{\mbox{Aut}}   \def\eps{\epsilon}
 \def\dim{\operatorname{dim}} 
\def\lk{\mbox{lk}}
\def\FF{\Bbb{F}}
\def\F{\Bbb{F}}
\def\sltwoc{\operatorname{SL}(2,\C)}
\def\tpm{[t^{\pm 1}]}

\def\endo{\mbox{End}}

\def\i{\iota}
\def\ol{\overline}
\def\ul{\underline}

\def\G{\Gamma}
\def\ct{\C\tpm}

\def\deg{\mbox{deg}}

\def\d{\dagger}
\def\ctau{\check{\tau}}
\def\pr{\operatorname{pr}}

\begin{document}

\title{Poincar\'e duality and degrees of twisted Alexander polynomials}
\author{Stefan Friedl}
\author{Taehee Kim}
\author{Takahiro Kitayama}
\address{Mathematisches Institut\\ Universit\"at zu K\"oln\\   Germany}
\email{sfriedl@gmail.com}

\address{Department of Mathematics\\
  Konkuk University\\
  Seoul 143--701\\
  Republic of Korea}

\email{tkim@konkuk.ac.kr}

\address{Research Institute for Mathematical Sciences\\
Kyoto University\\
Kyoto 606--8502\\
Japan}

\email{kitayama@kurims.kyoto-u.ac.jp}

\date{\today}



\begin{abstract}
We prove duality theorems for twisted Reidemeister torsions and twisted Alexander polynomials generalizing the results of Turaev. As a corollary we determine the parity of the  degrees of twisted Alexander polynomials of 3-manifolds in many cases.
\end{abstract}

\maketitle

\section{Introduction}
Let $N$ be a  3-manifold with empty or toroidal boundary. Here and throughout the paper, we will assume that all 3-manifolds are connected, compact and orientable.
We also assume that all manifolds and manifolds pairs are smooth. A \emph{homology orientation} of $N$ is an orientation of the real vector space
 $\oplus_{i\geq 0}H_*(N;\R)$. We denote by $\ort(N)$ the set of homology orientations of $N$. If $\w\in \ort(N)$,
 then we denote by $-\w$ the opposite homology orientation.

Given a space $X$ we denote by  $\HH_1(X)$ the first integral homology group viewed  as a multiplicative group.
Given a 3-manifold $N$ with empty or toroidal boundary we denote by $\spinc(N)$ the set of $\spinc$-structures on $N$. Recall that there exists a canonical free and transitive action of $\HH_1(N)$ on $\spinc(N)$ and that any $\spinc$-structure $\ss$ has a Chern class $c_1(\ss)\in \HH_1(N)$.

Let $N$ be a 3-manifold with empty or toroidal boundary and let  $\varphi\colon \pi_1(N)\to \gl(d,\FF)$ be a representation over a field $\FF$.
Let $\ss \in \spinc(N)$
and let $\w\in \ort(N)$. Building on ideas of Turaev (see \cite{Tu86,Tu90,Tu01})
we introduce a twisted torsion  invariant
\[ \tau(N,\varphi,\ss,\w)\in \FF \]
which has no indeterminacy (see Section~\ref{sec:twisted homology groups}). This invariant is the \emph{sign-refined torsion of $(N,\varphi)$ associated to the $\spinc$-structure $\ss$ and the homology orientation $\w$}. This invariant is a generalization of the sign-refined (untwisted)
torsion  associated to a $\spinc$-structure and a homology orientation introduced by Turaev \cite{Tu90,Tu01}.
In Lemma  \ref{lem:indet} we will see that the torsion invariant has the following formal property:
Let $(N,\varphi,\ss,\w)$ as above and let $\eps\in \{-1,1\}$ and $h\in \HH_1(N)$,
then
\be \label{equ:indet} \tau(N,\varphi,h\cdot \ss,\eps\cdot \w)=\eps^d\cdot \det(\varphi(h^{-1}))\cdot \tau(N,\varphi,\ss,\w)\in \FF. \ee
(Here and throughout this paper we  make use of the fact that the map $g\mapsto \det(\varphi(g))$ descends to
a map  $\HH_1(N)=H_1(\pi;\Z)\to \FF^\times:=\FF\sm \{0\}$.)
In particular, if $\varphi$ is an even-dimensional representation, then the torsion invariant does \emph{not} depend on the homology orientation.

Now suppose that $\FF$ is equipped with (possibly trivial) involution.
Note that
$g\mapsto \ol{\varphi(g^{-1})}^t$ is also a representation, which we refer to as  the \emph{dual} of $\varphi$ and which we denote by $\varphi^\d$.
The following is our main theorem of the paper.

\begin{theorem}\label{mainthmstrongintro}
Let $N$ be a  3-manifold with empty or toroidal boundary.
Let  $\varphi\colon \pi_1(N)\to \gl(d,\FF)$ be a representation
over a field $\FF$ with (possibly trivial) involution,
 let $\ss \in \spinc(N)$
and let $\w\in \ort(N)$. Suppose that $H_*^\varphi(\partial N;\F^d)=0$, then
\[ \ol{\tau(N,\varphi,\ss,\w)}=(-1)^{db_0(\partial N)}\det(\varphi(c_1(\ss)))\cdot  \tau(N,{\varphi^\d},\ss,\w). \]
\end{theorem}

The proof naturally splits up into parts:
\bn
\item Given any orientable $m$-manifold $N$ we relate in  Theorem \ref{thm:dualitygeneral}   the twisted Reidemeister torsion of $N$
to an appropriate twisted Reidemeister torsion of $(N,\partial N)$.
\item Given a 3-manifold $N$ we show in  Theorem \ref{thm:torsionboundary} how to express the twisted Reidemeister torsion of $(N,\partial N)$ again in terms of an appropriate twisted Reidemeister torsion of $N$.
 \en
The proof for both results rely heavily on the ideas of Turaev as presented in \cite{Tu86,Tu90,Tu01,Tu02}.

Our main motivation lies in the application of Theorem \ref{mainthmstrongintro} to the case of twisted Alexander polynomials. The set up is now as follows:
Let $N$ be a  3-manifold with empty or toroidal boundary. We write $\pi=\pi_1(N)$. Let $\a\colon \pi\to \gl(d,\KK)$ be a representation over a field $\KK$ and let $\phi\colon \pi\to F$ be an epimorphism onto a free abelian group. (Throughout the paper we will  view $F$ as a multiplicative group.)
We denote by $\KK(F)$ the quotient field of the group ring $\KK[F]$.
Note that $\a$ and $\phi$ give rise to the tensor representation $\a\otimes \phi\colon \pi \to \gl(d,\KK(F))$.
We  then define the \emph{twisted Alexander polynomial} of $(N,\a,\phi)$ to be the set:
\[ \tau(N,\a\otimes \phi):=\{\tau(N,\a\otimes \phi,\ss,\w)\,|\, \ss \in \spinc(N)\mbox{ and }\w\in \ort(N)\}\subset \KK(F).\]
We call each $\tau(N,\a\otimes \phi,\ss,\w)$ a \emph{representative} of $\tau(N,\a\otimes \phi)$.
Note it follows from (\ref{equ:indet}) that representatives of $\tau(N,\a\otimes \phi)$
differ by multiplication by
 an element in
 \[ \{ \eps\cdot f^d\cdot a\,|\, \eps=1\mbox{ or }\eps=(-1)^d, f\in F, a\in \det(\a(\pi))\}. \]
The invariant $\tau(N,\a\otimes \phi)$ is a variation on   the twisted invariants introduced by Lin \cite{Li01}, Wada \cite{Wa94},
Kitano \cite{Ki96}, Kirk and Livingston \cite{KL99}, Kitayama \cite{Kiy08} and many other authors.
In the literature the invariant $\tau(N,\a\otimes \phi)$ is often referred to as
`twisted Alexander polynomial', `twisted Reidemeister torsion', or (in the case of link complements) as
`Wada's invariant'.
We refer to Section \ref{section:twialex} and to \cite{FV10} for more information on twisted torsion and twisted Alexander polynomials of 3-manifolds. Note that $\tau(N,\a\otimes \phi)$ is by definition a rational function,
but in many cases, e.g. if $\mbox{rank}(F)\geq 2$ or if $\a$ is `sufficiently non-abelian', the invariant
 $\tau(N,\a\otimes \phi)$ does actually lie in $\KK[F]\subset \KK(F)$. We refer to Section \ref{section:polynomial} for details.
Finally, we say that an epimorphism $\phi\colon \pi\to F$ is \emph{admissible} if $\phi$ restricted to any boundary component of $N$ is non-trivial.

We say that two representations $\varphi,\psi\colon \pi\to \gl(d,\FF)$ are \emph{conjugate} if there exists a matrix $P$ such that
$\varphi(g)=P\psi(g)P^{-1}$ for all $g\in \pi$. In the following we will be interested in representations which are conjugate to their duals.
Note that by   \cite[Lemma~3.1]{HSW10} a representation $\a\colon \pi\to \gl(d,\FF)$ is conjugate to its dual if
and only if there exists a non-degenerate sesquilinear
 form on $\FF^d$ such that $\a$ acts by isometries.
The following gives a list of representations which are conjugate to their duals:
\bn
\item Orthogonal representations.
\item Unitary representations.
\item Representations to $\sl(2,\C)$ \emph{if we view $\C$ as equipped with the trivial involution}.
\item The adjoint representation of an $\sl(2,\C)$-representation.
\en
We refer to Section \ref{section:sl2} for more information.

In the following, given a field $\KK$ with (possibly trivial) involution and a free abelian group $F$ we equip $\KK(F)$ with the involution induced by the involution on $\KK$ and by $\ol{f}=f^{-1}$ for any $f\in F$.
With this convention Theorem \ref{mainthmstrongintro}
now gives rise to the following theorem on twisted Alexander polynomials:

\begin{theorem}\label{mainthm:twialexintro}
Let $N$ be a  3-manifold with empty or toroidal boundary, let  $\a\colon \pi_1(N)\to \gl(d,\KK)$ be a representation
over a field $\KK$ with (possibly trivial) involution, and let $\phi\colon H_1(N;\Z)\to F$
be an admissible epimorphism onto a free abelian group.
Suppose that $\a$ is conjugate to its dual.
Then for any representative $\tau$ of $\tau(N,\a\otimes \phi)$ we have
\[ \ol{\tau}=(-1)^{db_0(\partial N)}\cdot \det(\a(g))\cdot\phi(g)^d\cdot \tau \]
for some $g\in \HH_1(N)$. Furthermore if $N$ is closed, then  there exists a representative $\tau$ of $\tau(N,\a\otimes \phi)$ such that
\[ \ol{\tau}= \tau\in \KK(F).\]
\end{theorem}

Using further results of Turaev one can obtain a more precise statement if  $N$ is the exterior of a link in a homology 3-sphere.
Let $L=L_1\cup \dots \cup L_m \subset N$ be an oriented, ordered $m$-component link in a closed 3-manifold. We then denote by $N_L=N\sm \nu L$ the exterior of $L$.
If $N$ is a $\Z$-homology sphere, then the oriented ordered meridians of $L$ form a basis for $H_1(N_L;\Z)$, in particular we have a canonical isomorphism $\Z^m\to H_1(N_L;\Z)$.
We will use this canonical isomorphism to identify $\HH_1(N_L)$ with the free abelian multiplicative group generated by $t_1,\dots,t_m$.
Following Turaev (see \cite[p.~76]{Tu02}) we define a \emph{charge} of $L$ to be a word $t_1^{n_1}\cdot \dots\cdot t_m^{n_m}$ such that for any $i=1,\dots,m$ we have
\[ n_i\equiv 1+\sum_{j\ne i} \lk(L_i,L_j) \mod 2.\]

\begin{theorem}\label{mainthm:twialexlinkintro}
Let $L$ be a link in a $\Z$-homology sphere $N$, let  $\a\colon \pi_1(N_L)\to \gl(d,\KK)$ be a representation
over a field $\KK$ with (possibly trivial) involution, and let $\phi\colon H_1(N;\Z)\to F$
be an admissible epimorphism onto a free abelian group.
Suppose that $\a$ is conjugate to its dual.
Then for any  representative $\tau$ of $\tau(N_L,\a\otimes \phi)$ there exists a charge $c$ such that
\[ \ol{\tau}=(-1)^{db_0(N)}\cdot  \det(\a(c))\cdot\phi(c)^d\cdot \tau\in \KK(F).\]
Conversely,  for any charge $c$ there exists a representative $\tau$ of $\tau(N_L,\a\otimes \phi)$ such that
\[ \ol{\tau}=(-1)^{db_0(N)}\cdot  \det(\a(c))\cdot \phi(c)^d\cdot \tau\in \KK(F).\]
\end{theorem}

\begin{remark}
Our  theorems  generalize  several earlier duality results on torsion invariants which we now list:
 \bn
 \item Seifert \cite{Se35}, Fox and Torres (\cite{To53} and \cite[Corollary~3]{FT54}),  Milnor \cite{Mi62}
 and Turaev  \cite[Theorem~1.7.1~and~p.~141]{Tu86} proved results about symmetries and degrees of untwisted Alexander polynomials.
 \item  Kitano \cite[Theorem~B]{Ki96} showed that if $N$ is a knot complement and $\a\colon \pi_1(N)\to \mbox{SO}(n)$ a special orthogonal representation, then for any representative $\tau$ of the corresponding Reidemeister torsion we have $\ol{\tau}=\eps t^{kn}\tau$ for some $k\in \Z$ and $\eps\in \{-1,1\}$.
 \item Kirk and Livingston \cite[Corollary~5.2]{KL99} showed that if $N$ is a closed 3-manifold and $\a$ a unitary representation over a subfield $\KK\subset \C$, then for any representative $\tau$ of  $\tau(N,\a\otimes \phi)$ we have $\ol{\tau}=u\tau$ where $u$ is a unit in $\KK[F]$. (See also \cite[Section~2]{FK06}.)
 \item Kitayama \cite[Theorem~5.9]{Kiy08} proved Theorem \ref{mainthm:twialexintro} for the case of a knot complement.
 \item Hillman, Silver and Williams \cite[Theorem~3.2]{HSW10} showed that under the same assumptions of Theorem \ref{mainthm:twialexintro},
 for any representative $\tau$ of $\tau(N,\a\otimes \phi)$  we have $\ol{\tau}=\eps fd\tau$
 for some $\eps\in \{-1,1\}$,  $f\in F$ and $d\in \det(\a(\pi_1(N)))$.
 \item Dubois and Yamaguchi \cite[Theorem~14]{DY09} and \cite[Remark~16]{DY09} studied the twisted Alexander polynomial corresponding to the adjoint representation of an $\sl(2,\C)$-representation.
\en
 \end{remark}

Let $N$ be a 3-manifold with empty or toroidal boundary and let  $\phi\in H^1(N;\Z)=\hom(\pi_1(N,\Z))$.
The \emph{Thurston norm} of $\phi$ is defined as
 \[
x(\phi):=\min \{ \chi_-(S)\, | \, S \subset N \mbox{ properly embedded surface dual to }\phi\}.
\]
Here, given a surface $S$ with connected components $S_1\cup\dots \cup S_k$, we define
$\chi_-(S)=\sum_{i=1}^k \max\{-\chi(S_i),0\}$. We refer to \cite{Th86} for details.
Now let $\a\colon \pi_1(N)\to \gl(d,\KK)$ be a  representation over a field $\KK$ with involution. We identify $\KK[\Z]$ with $\KK\tpm$ and $\KK(\Z)$ with $\KK(t)$.
We then view $\tau(N,\a\otimes \phi)$ as an element in $\KK(t)$.
Given $p=\sum_{i=k}^la_it^i\in \kt$ with $a_k\ne 0$ and $a_l\ne 0$ we define $\deg(p)=l-k$.
Given a non-zero fraction $f=\frac{p}{q}\in \KK(t)$ we define $\deg(f)=\deg(p)-\deg(q)$.
Note  that $\deg(\tau(N,\a\otimes \phi))\in \Z$ is well-defined, i.e. independent of a representative of $\tau(N,\a\otimes \phi)\in \KK(t)$.

We can now prove the following theorem. Note that it could not be proved with the previously known duality results for torsion.

\begin{theorem}\label{thm:evendegreeintro}
Let $N$ be an irreducible 3-manifold with empty or toroidal boundary such that $N\ne S^1\times D^2$, let  $\a\colon \pi_1(N)\to \gl(d,\KK)$ be a representation over a field $\KK$ with involution, and let $\phi\colon \pi_1(N)\to \Z$ be an admissible epimorphism. If $\a$ is conjugate to its dual
and if $\tau(N,\a\otimes \phi)\ne 0$, then
 \[ \deg(\tau(N,\a\otimes \phi))\equiv d\cdot x(\phi)\,\mod\,2.\]
\end{theorem}

\begin{remark}
\bn
\item In \cite{FK06} it is shown, that under the assumptions of Theorem~\ref{thm:evendegreeintro} we have the following inequality:
 \[ \deg(\tau(N,\a\otimes \phi))\leq d\cdot x(\phi).\]
\item
This result is also closely related to earlier results by two of the authors on the degrees of the Cochran-Harvey `higher order Alexander polynomials'. We refer to \cite{Co04,Ha05,FK08a} for details.
In fact the methods of this paper can be extended in a fairly straightforward way to recover the main theorem of \cite{FK08a} and in fact to somewhat generalize the results of \cite{FK08a}. Indeed,
one can now prove \cite[Theorem~1.2]{FK08a} without the assumption that the ordinary Alexander polynomial is non-zero.
\item Theorem~\ref{thm:evendegreeintro} is optimal in various ways. For example by \cite{Mo11a}, there exists a 3-manifold $N$ with boundary, $\phi\in H^1(N;\Z)$ and  an even dimensional representation  $\a$ such that
     the formula in Theorem \ref{thm:evendegreeintro} does not hold modulo four.
      By \cite{Mo11b} similar examples also exist in the closed case.
    Also, there exist closed 3--manifolds and representations $\a$ which  are not conjugate to their duals, such that the corresponding Reidemeister torsions do not have even degree. For example, if $M$ is the 0--framed surgery on the knot $11_{412}$, then there exists an epimorphism $\a:\pi_1(M)\to S_5$ such that the corresponding representation
\[ \pi_1(M)\xrightarrow{\a}S_5\to \aut\{ (v_1,\dots,v_5)\in \F_5^5\,|\, v_1+\dots+v_5=0\}=\gl(\F_5,4)\]
(where $S_5$ acts by permutation on the given vector space) has a corresponding twisted Alexander polynomial in $\F_5\tpm$ of degree 13.
\en
\end{remark}

Of particular interest in recent years has been the study of twisted Alexander polynomials corresponding to
$\sl(2,\C)$-representations, see e.g. \cite{Mo08,DY09,KmM10,FJ11,DFJ11,Ya11}.
We  conclude with the following result which generalizes \cite[Corollary~3.4]{HSW10}.
We say $\tau\in \C(t)$  is a \emph{loose representative} for $\tau(N,\a\otimes \phi)$
if for an (and then for any) $\ss\in \spinc(N)$ and $\w\in \ort(N)$ there exists a $k\in \Z$
such that $\tau=t^k\cdot \tau(N,\a\otimes \phi,\ss,\w)$.

\begin{theorem}\label{thm:sltwocintro}
Let $N$ be a 3-manifold with empty or toroidal boundary, let  $\a\colon \pi_1(N)\to \sl(2,\C)$ be an irreducible representation
and  let $\phi\colon \pi_1(N)\to \Z$ be an epimorphism. Then there exists a loose representative $\tau$ of $\tau(N,\a\otimes \phi)\in \C(t)$ of the form
\[\sum_{i=0}^{l} a_i(t^{-i}+t^{i})\]
for some $a_0,\dots,a_l\in \C$. In particular $\tau(N,\a\otimes \phi)$ is of even degree.
\end{theorem}

This paper is organized as follows. In Sections \ref{sec:Euler structures} and \ref{sec:twisted homology groups} we review some basic materials such as Euler structures and twisted homology groups. In Sections \ref{sec:duality for Euler structures} and \ref{sec:duality for torsion} we study duality for Euler structures and torsions of manifolds. Theorem~\ref{mainthmstrongintro} is proved in Section~\ref{sec:proof} and Theorems \ref{mainthm:twialexintro}, \ref{mainthm:twialexlinkintro}, \ref{thm:evendegreeintro}, and \ref{thm:sltwocintro} are proved in Section~\ref{section:twialex}. Finally in Section~\ref{section:polynomial} we prove that in many cases twisted Alexander polynomials of 3-manifolds are in fact Laurent polynomials.

\subsection*{Conventions.} All manifolds are assumed to be connected, orientable and compact.
All CW-complexes are assumed to be finite and connected.
By a field we will always mean a commutative field. By a field with involution we also allow the case that the involution is trivial. We usually think of free abelian groups as multiplicative groups.
A basis of a vector space is always understood to be an ordered basis.

\subsection*{Acknowledgment.}
We wish to thank  Jae Choon Cha, J\'er\^ome Dubois and Dan Silver  for helpful discussions and conversations. We also thank Takayuki Morifuji, Masaaki Suzuki and Teruaki Kitano who organized the RIMS seminar `Twisted topological invariants and topology of low-dimensional manifolds' in 2010 at Akita, Japan where the authors started this project. The second author was supported by the National Research Foundation of Korea (NRF) grant funded by the Korean government (MEST) (No. 2011-0003357 and No.
2011-0001565). The last author was supported by Research Fellowships for Young Scientists of the Japan Society for the Promotion of Science.

\section{Euler structures of manifolds}
\label{sec:Euler structures}
\subsection{Euler structures}

Let $X$ be a finite CW-complex of dimension $m$ and let $Y$ be a proper subcomplex such that $\chi(X,Y)=0$.
In the following we recall the definition of Euler structures of $(X,Y)$ as introduced by  Turaev (see \cite{Tu90,Tu01}).

We denote by $p \colon \wti{X}\to X$ the universal covering of $X$ and we write $\wti{Y}:=p^{-1}(Y)$.
An \emph{Euler lift $c$} is a set of cells in $\wti{X}$  such that each $i$-cell of $X\sm Y$ is covered by precisely one of the cells  in the
 Euler lift.

Using the canonical left action of $\pi=\pi_1(X)$ on $\wti{X}$ we  obtain
a free and transitive action of $\pi$ on the set of cells of $\wti{X}\sm \wti{Y}$ lying over a fixed cell in $X\sm Y$.
If $c$ and $c'$  are two Euler lifts, then we can order the cells such that $c=\{c_{ij}\}$ and $c'=\{c_{ij}'\}$ and such that for each $i$ and $j$ the cells $c_{ij}$ and $c_{ij}'$ lie over the same $i$-cell in $X\sm Y$.
In particular  there
 exist  unique $g_{ij}\in \pi$
such that
$c_{ij}'=g_{ij}\cdot c_{ij}$.
We now write $\HH=\HH_1(X)$ and we denote the projection map $\pi\to \HH$ by $\Psi$.
We define
\[ c'/c:= \prod\limits_{i=0}^m\prod\limits_{j}\Psi(g_{ij})^{(-1)^i}\in  \HH.\]
We say that $c$ and $c'$ are \emph{equivalent} if $c'/c\in \HH$ is trivial.
An equivalence class of Euler lifts will be referred to as an \emph{Euler structure}.
We  denote by $\eul(X,Y)$ the set of Euler structures.
Given a Euler lift $c$ we sometimes denote by $[c]$ the Euler structure in $\eul(X,Y)$ represented by $c$ .
If $Y=\emptyset$ then we will also write $\eul(X)=\eul(X,Y)$.

Since $X\sm Y\ne \emptyset$ we will now see that  $\HH$ acts freely and transitively on the equivalence classes of Euler structures.
Indeed, given $g\in  \HH$ and $e\in \eul(X,Y)$ we define $g\cdot e$ as follows: pick a representative $c$ for $e$ and pick $\wti{g}\in \pi_1(X)$ which represents $g$, then act on  one $i$-cell of $c$ by $ g^{(-1)^i}$.
 The resulting Euler lift represents an element in $\eul(X,Y)$ which is independent of the choice of the cell. We denote by  $ g\cdot e$
 the Euler structure represented by this new Euler lift.
 Clearly $( g\cdot e)/e= g$.

\subsection{Euler chains}

Let $X$ be a finite CW-complex of dimension $m$ and let $Y$ be a proper subcomplex such that $\chi(X,Y)=0$.
 We denote by $A$ the set of cells of $X\sm Y$.
Following Turaev \cite{Tu90} we define an \emph{Euler chain} in $(X,Y)$
to be  a one-dimensional singular chain $\xi$ in $X$ such that
\[ \partial \xi=\sum_{a\in A}(-1)^{\operatorname{dim}(a)}p_a, \]
where for each $a$, $p_a$ is a point in $a$. Given  two Euler chains
$\xi$ and $\eta$ with $\partial \xi=\sum_{a\in A}(-1)^{\dim(a)}p_a$ and $\partial \eta=\sum_{a\in A}(-1)^{\dim(a)}q_a$ we pick for each $a\in A$ a path $x_a:[0,1]\to a$ from $p_a$ to $q_a$. We then define $ \xi-\eta$ to be the homology class in $H_1(X;\Z)$ represented by the cycle
\[ \xi-\eta+\sum_{a\in A}(-1)^{\dim(a)}x_a.\]
We say that the Euler chains $\xi$ and $\eta$ are \emph{homologous} if $\xi-\eta=0\in H_1(X;\Z)$.
Note that there is a canonical free and transitive action by $\HH_1(X)=H_1(X;\Z)$ on the set of homology classes of Euler chains
defined by $[h][\xi] = [\xi + h]$ for $[h]\in \HH_1(X)$.

To an Euler lift we can associate an Euler chain as follows:
pick a base point of the universal cover $\wti{X}$ and pick a path from that base point to a point in each cell of $\wti{X}$ which lies in the  Euler lift. Taking the alternating sum (where the sign comes from the parity of the dimension of the cell) of the projection of these paths to $X$ we obtain an Euler chain.
It is straightforward to  see that this  map induces an  $\HH_1(X)$-equivariant bijection
\[ \eul(X,Y)\to \{\mbox{homology classes of Euler chains}\}.\]
We refer to \cite[Section~1.3]{Tu90} for more details. We will use this bijection to identify Euler structures with homology classes of Euler chains. We will freely go back and forth between these two notions.

\subsection{Euler structures and subdivisions}\label{section:eulersubdivision}

Let $(X',Y')$ be a cellular subdivision of $(X,Y)$.
There exists a canonical $\HH_1(X)$-equivariant map
\[ \s \colon \eul(X,Y)\to \eul(X',Y')\]
 which is defined as follows:
 Let $e\in \eul(X,Y)$ and pick an Euler lift for $(X,Y)$ which represents $e$. There exists a unique Euler lift for $(X',Y')$ such that the cells in the Euler lift of $(X',Y')$ are contained in the cells of the Euler lift of $(X,Y)$. We then denote by $\s(e)$ the Euler structure represented by this Euler lift.
This map agrees with the map defined by Turaev \cite[Section~1.2]{Tu90}.

\subsection{Euler structures of smooth manifolds}

We will now quickly recall the definition of Euler structures on smooth manifolds. We refer to \cite[Sections~2.1~and~2.2]{Tu90} for full details.

Let $N$ be a manifold and $\partial_0N \subset \partial N$  a union of components of $\partial N$ such that $\chi(N,\partial_0 N)=0$.
(In all our later applications we will either take $\partial_0N=\partial N$ or $\partial_0 N=\emptyset$.)
We write $\HH=\HH_1(N)$.
A \emph{triangulation} of  $N$ is a pair $(X,t)$
where $X$ is a  simplicial complex and $t:|X|\to N$ is a homeomorphism. Note that $t^{-1}(\partial_0 N)$ is a simplicial subspace of $X$.
Throughout this paper we write $Y:=t^{-1}(\partial_0 N)$.  We will for the most part suppress $t$ from the notation.
Following \cite[Section~I.4.1]{Tu90} we consider the projective system of sets $\{\eul(X,Y)\}_{(X,t)}$ where $(X,t)$ runs over all $C^1$-triangulations of $N$
and where the maps are the $\HH$-equivariant bijections between these sets induced either by $C^1$-subdivisions or by smooth isotopies in $N$.

We now define $\eul(N,\partial_0 N)$ by identifying the sets $\{\eul(X,Y)\}_{(X,t)}$ via these bijections. We refer to $\eul(N,\partial_0 N)$ as the set of Euler structures on $(N,\partial_0 N)$.
Note that for a $C^1$-triangulation $X$ of $N$ we get a canonical $\HH$-equivariant bijection $\eul(X,Y)\to \eul(N,\partial_0 N)$.

\subsection{Smooth Euler structures}\label{section:smootheuler}

Let $N$ be a compact manifold and $\partial_0 N$ a union of components of $\partial N$ such that $\chi(N,\partial_0 N)=0$.
Following Turaev (see \cite[Section~5.1]{Tu90})  we
define a  \emph{regular  vector field} on  $(N,\partial_0 N)$ to be a nowhere vanishing vector field on $N$ which points inwards on $\partial_0 N$ and outwards on $\partial N\sm \partial_0 N$.
Two such vector fields are called \emph{homologous} if for some point $x\in \operatorname{Int}(N)$ the restrictions of the vector fields to $N\sm x$ are homotopic in the set of all regular vector fields on $N\sm x$.
The set of homology classes of regular vector fields is called $\vect(N,\partial_0 N)$.
Turaev \cite[p.~639]{Tu90} showed that $\vect(N,\partial_0 N)$ admits a canonical, free and transitive action by $\HH_1(N)$,
and by  \cite[Section~6.1~and~6.6]{Tu90}  there exists
a canonical
 $\HH_1(N)$-equivariant map
$ca_N\colon \eul(N,\partial_0 N)\to \vect(N,\partial_0 N)$.
We will recall the definitions in the proof of Lemma \ref{lem:canmaps}.

\subsection{Homology orientations}

Let $(X,Y)$ be a pair of topological spaces. A \emph{homology orientation} of $(X,Y)$ is an orientation of the real vector space
 $\oplus_{i\geq 0}H_*(X,Y;\R)$. We denote by $\ort(X,Y)$ the set of homology orientations of $(X,Y)$. Given $\w\in \ort(X,Y)$ we denote by $-\w$ the opposite homology orientation.

\section{Twisted  homology groups and twisted torsion}
\label{sec:twisted homology groups}

\subsection{Torsion of based complexes} \label{section:torsion}
In this subsection, we quickly recall the definition of the torsion of a based complex.
We refer to Milnor's classic paper \cite{Mi66}, Nicolaescu's monograph \cite{Ni03} and Turaev's books \cite{Tu01,Tu02} for details.
Note that we follow Turaev's convention; Milnor's definition gives the multiplicative inverse of the torsion we consider.

Throughout this section, let $\FF$ be a field.
Let $V$ be a vector space over $\FF,$ and  let $x=(x_1,\dots,x_m)$ and $y=(y_1,\dots,y_m)$ be two ordered bases for $V$. Then we
can write $x_i=\sum_{j=1}^na_{ij}y_j,$ and we
define $ [x/y] =  \det(a_{ij})$.

Now let
\[ 0\to C_m \xrightarrow{\partial_{m-1}} C_{m-1}\xrightarrow{\partial_{m-2}} \dots \xrightarrow{\partial_1} C_{1} \xrightarrow{\partial_0 } C_0\to 0\]
be a complex of $\FF$-vector spaces. We write $H_i =H_i(C).$ For each $i,$ we pick an ordered basis $c_i$ for $C_i$ and an ordered basis $h_i$ for $H_i.$

We write $B_i=\im\{\partial_i\colon C_{i+1}\to C_i\},$ and we pick an ordered  basis $b_i$ for $B_i$.
Finally, we pick an ordered set of vectors $b_i'$ in $C_{i+1}$ such that $\partial_i(b_{i}')=b_i$ as ordered sets.
By convention, we define $b_{-1}'$ to be the empty set.
Note that for $i=0,\dots,m,$ the ordered set $b_ih_ib'_{i-1}$ defines an ordered basis for $C_i$.
We now define the \emph{torsion} of the based complex $C$ as
\[\tau =  \prod_{i=0}^m [b_ih_ib'_{i-1}/c_i]^{(-1)^{i+1}} \in \FF^\times.\]
An elementary argument shows that $\tau$ does not depend on the choice of $b_0,\dots,b_{m-1},$ and it does not depend on the choice of  lifts $b_0',\dots,b_{m-1}',$ see for example \cite[Section~1]{Tu01}. Put differently, this number only depends on the choice of the  complex and the choice of the ordered bases for $C_*$ and $H_*.$ We henceforth denote this invariant by
$\tau(C_*,c_*,h_*)\in \FF^\times.$ If $C_*$ is acyclic; i.e., if $H_*(C)=0,$ then we just write $\tau(C_*,c_*)\in \FF^{*}.$

Note that if $c_*'$ is an ordered basis obtained from $c_*$ by swapping two basis vectors, then
it is straightforward to see that
\[ \tau(C_*,c_*',h_*)=- \tau(C_*,c_*,h_*)\in \FF^\times,\]
furthermore if $c_*'$ is an ordered basis obtained from $c_*$ by multiplying one basis vector of $C_i$ by $f\in \FF^\times$, then
\[ \tau(C_*,c_*',h_*)=f^{(-1)^i}\cdot \tau(C_*,c_*,h_*)\in \FF^\times.\]

We will frequently need a renormalization of torsion which is due to Turaev (see \cite{Tu86}
and \cite[Section~3]{Tu90}).
Given a chain complex $C_*$ we define
\[ \ba{rcl} \a_i(C_*)&=& \sum_{j=0}^i (-1)^{j-i} \dim\,C_j,\\[1mm]
\b_i(C_*)&=&\sum_{j=0}^i (-1)^{j-i}\dim\,H_j(C_*),\\[1mm]
\eta(C_*)&=&\sum_{i} \a_i(C_*)\b_i(C_*).\ea \]
We then define
\[ \check{\tau}(C_*,c_*,h_*):= (-1)^{\eta(C_*)}\tau(C_*,c_*,h_*).\]

\subsection{Twisted chain complexes}
Let $(X,Y)$ be a pair of topological spaces.
We write $\pi=\pi_1(X)$. Let $\varphi\colon \pi\to \gl(d,\FF)$ be a representation
over a field $\FF$.
Denote by $ p \colon \widetilde{X}\to X $ the universal covering  of $X$. We write
$\wti{Y}:=p^{-1}(Y)$.
The  chain complex $C_*(\widetilde{X},\wti{Y})$ is a left $\Z[\pi]$--module via deck transformations.
Using the natural involution $g\mapsto g^{-1}$ on the group ring $\Z[\pi]$, we can also view $C_*(\widetilde{X},\wti{Y})$ as a right $\Z[\pi]$--module. By viewing elements of $\F^d$ as column vectors the representation $\varphi$ gives rise to a left $\Z[\pi]$-module structure on $\F^d$.
We now obtain the twisted  chain complex
\[ C_*^{\varphi}(X,Y;\F^d):=C_*(\widetilde{X},\wti{Y})\otimes_{\Z[\pi]}\F^d,\]
and we denote its homology groups by $H_*^\varphi(X,Y;\F^d)$.
 When $\varphi$ is  understood we will drop it  from the notation.

 \begin{remark}
Let $(X,Y,\varphi)$ as above. We can view $\F^d$ as a right $\Z[\pi$]-module via the action
$g\mapsto (v\mapsto v\a(g))$ where we now view elements of $\F^d$ as row vectors.
We can now consider the twisted  chain complex
\[ D_*^{\varphi}(X,Y;\F^d):=\F^d\otimes_{\Z[\pi]}C_*(\widetilde{X},\wti{Y}).\]
This point of view is taken in various papers in the literature (see e.g. \cite{KL99,Tu01,Tu02}).
We now denote by $\psi$ the representation $g\mapsto \varphi(g^{-1})^t$.
It is straightforward to verify that the map $\sigma\otimes v\mapsto v^t\otimes \sigma$ gives rise to a chain complex isomorphism $C_*^{\varphi}(X,Y;\F^d)\to D_*^{\psi}(X,Y;\F^d)$.
 \end{remark}

For future reference, when $(X,Y)$ is a pair of topological spaces, then we define $\a_i(X,Y), \b_i(X,Y)$ and $\eta(X,Y)$ to be the invariants of the chain complex $C_*(X,Y;\R)$.
Furthermore, when  $\varphi\colon \pi\to \gl(d,\FF)$ is a representation, then we define
$\a_i^\varphi(X,Y)$, $\b_i^\varphi(X,Y)$ and $\eta^\varphi(X,Y)$ to be the invariants of the chain complex $C_*^\varphi(X,Y;\F^d)$.
As usual we drop $Y$ from the notation when $Y=\emptyset$.
We will several times make use of the following basic observation:
\[ \a_i^\varphi(X,Y)=d\cdot \a_i(X,Y).\]

\subsection{Twisted  torsion of CW-complexes}\label{section:twitorsion}\label{section:twitau}

Now let $X$ be a finite CW-complex and $Y$ a proper subcomplex, $\varphi\colon \pi_1(X)\to \gl(d,\FF)$ a representation over a field $\FF$,
$e\in \eul(X,Y)$ an Euler structure and $\w$ a homology orientation for $(X,Y)$.

If the complex $C_*^\varphi(X,Y;\F^d)$ is not acyclic, i.e. if $H_*^\varphi(X,Y;\F^d)\neq 0$, then we define $\tau(X,Y,\varphi,e,\w):=0$.
Now suppose that $H_*^\varphi(X,Y;\F^d)=0$. (It is worth noting, that it is a necessary condition that the Euler characteristic $\chi(X,Y)$ is zero.)
We pick an  Euler lift which represents $e$ and we pick an orientation for each cell in the Euler lift and we pick an ordering for the cells in the Euler lift. We then denote the set of oriented $i$-cells by $c_{i1},\dots,c_{is_i}$. Finally  we equip $\F^d$ with the canonical ordered basis $v_1=(1,0,\dots,0),\dots,v_n=(0,\dots,0,1)$.
For each $i$ we can then view
\[  C_i(\wti{X},\wti{Y})\otimes_{\Z[\pi]}\F^d\]
 as a based $\FF$-vector space via the basis
\[ (c_{i1}\otimes v_1,\dots,c_{i1}\otimes v_d,\dots, c_{is_i}\otimes v_1,\dots,c_{is_i}\otimes v_n).\]
In the following we will, by slight abuse of notation, denote this basis by $\{c_{ij}\otimes v_k\}$ or $\{c_*\otimes v_*\}$.
(Recall that here and throughout the paper it is understood that a basis is in fact an ordered basis.)

We now also pick a basis $h_*$ for $H_*(X,Y;\R)$ which represents $\w\in \ort(X,Y)$ and we denote by $\ol{c}_{ij}$ the images of the cells $c_{ij}$ under the projection maps $\wti{X}\to X$.

We now  define
\[ \tau(X,Y,\varphi,e,\w):=\check{\tau}\big(C_*^\varphi(X,Y;\F^d),\{c_{ij}\otimes v_k\}\big)\cdot\sign\big(\check{\tau}(C_*(X,Y;\R),\{\ol{c}_{ij}\},h_*)\big)^d.\]
(Note that by assumption $H_*^\varphi(X,Y;\F^d)=0$, in particular $\eta(C_*^\varphi(X,Y;\F^d))=0$.)
We refer to this invariant as the \emph{sign-refined twisted torsion of the pair $(X,Y)$ with respect to the Euler structure $e$.}
This invariant can be viewed as a twisted analogue of the sign-refined torsion introduced by Turaev \cite{Tu86,Tu90}.
When $Y=\emptyset$, then we drop $Y$ from the notation, i.e. we define $\tau(X,\varphi,e,\w):=\tau(X,\emptyset,\varphi,e,\w)$.

The following well-known lemma now summarizes a few key properties of the sign-refined twisted torsion.

\begin{lemma}\label{lem:indet}
Let $(X,Y,\varphi,e,\w)$ as above. Then the following hold:
\bn
\item The invariant $\tau(X,Y,\varphi,e,\w)\in \FF$  is well--defined (i.e. independent of the choice of the  Euler lift and the ordering of the cells), put differently, $\tau(X,Y,\varphi,e,\w)$  has no indeterminacy.
\item Let $\eps\in \{\pm 1\}$ and $g\in \HH_1(X)$, then
\[ \tau(X,Y,\varphi, g\cdot e,\eps\cdot  \w)=\eps^d \cdot \det(\varphi(g^{-1}))\cdot \tau(X,Y,\varphi,e,\w).\]
\item If $\psi$ is conjugate to $\varphi$, then $\tau(X,Y,\varphi,e,\w)=\tau(X,Y,\psi,e,\w)$.
\en
\end{lemma}

\begin{proof}
We denote by $p \colon \wti{X}\to X$ the universal covering  of $X$ and write $\wti{Y}:=p^{-1}(Y)$. We also  we write $\pi=\pi_1(X)$.
Let $c_{ij}$ be an ordered, oriented collection of cells in $\wti{X}$  which represents $e$. Recall that
$ g\cdot e$ is represented by the ordered, oriented collection of cells  $\{c'_{ij}\}$ which is obtained from $\{c_{ij}\}$ by acting on one $i$-cell by $h:=g^{(-1)^i}$ using the
 canonical left $\pi$-action on $\wti{X}$.
On the other hand, in the chain complex $C_*(\wti{X},\wti{Y})\otimes_{\Z[\pi]}\F^d$ we view $C_*(\wti{X},\wti{Y})$ as a right $\Z[\pi]$-module using the involution. It follows that
\[  h\cdot c_{ij} \otimes v =c_{ij}\cdot h^{-1} \otimes v=c_{ij}\otimes  \varphi(h^{-1})(v).\]
It now follows easily, that
\[ \ba{rcl} \tau(C_*^\varphi(X,Y;\F^d),\{c'_{ij}\otimes v_k\})&=& \det(\varphi(h^{-1}))^{(-1)^i}\cdot \tau(C_*^\varphi(X,Y;\F^d),\{c_{ij}\otimes v_k\})\\
&=&\det(\varphi(g^{-1}))\cdot
\tau(C_*^\varphi(X,Y;\F^d),\{c_{ij}\otimes v_k\}), \mbox{ and }\\
 \sign(\tau(C_*(X,Y;\R),\{\ol{c'_{ij}}\},h_*))&=&\sign(\tau(C_*(X,Y;\R),\{\ol{c_{ij}}\},h_*)).\ea \]
 (Here $\{\ol{c'_{ij}}\}$ and $\{\ol{c_{ij}}\}$ denote of course the projections of the obvious cells to the base spaces.)
Finally, if $\{c'_{ij}\}$ is obtained from $\{c_{ij}\}$ by swapping two cells of the same dimension,  then the first torsion changes by a factor $(-1)^d$ and the latter changes by a factor $-1$.
Statements (1) and (2) now follow easily from the above observations. The third statement is straightforward to prove.
\end{proof}

\subsection{Euler structures and subdivisions}

Let $(X',Y')$ be a cellular subdivision of $(X,Y)$. By the discussion of Section \ref{section:eulersubdivision}
there  exists a canonical $\HH_1(X)$-equivariant map
\[ \s \colon \eul(X,Y)\to \eul(X',Y'). \]
A representation of $X$ and a homology orientation for $(X,Y)$ induce canonically a representation for $X'$ and a homology orientation for $(X',Y')$ which we denote by the same symbols.
The following lemma is now \cite[Lemma~3.2.3]{Tu90}:

\begin{lemma}\label{lem:subdiv} Let $e\in \eul(X,Y)$, $\varphi \colon \pi_1(X)\to \gl(d,\F)$ a representation and $\w$ a homology orientation of $(X,Y)$.
Then the following holds:
\[ \tau(X',Y',\varphi,\s(e),\w)=\tau(X,Y,\varphi,e,\w).\]
\end{lemma}

\subsection{Twisted  torsion of manifolds}

Let $N$ be a manifold and $\partial_0N \subset \partial N$  be a union of components of $\partial N$ such that $\chi(N,\partial_0 N)=0$.
Let  $\varphi\colon \pi_1(N)\to \gl(d,\FF)$ be a representation over a field $\FF$,
$e$ an Euler structure and $\w$ a homology orientation for $(N,\partial_0 N)$.

We write $\HH=\HH_1(N)$.
Recall that for any $C^1$-triangulation $f:X \to N$ we get a canonical $\HH$-equivariant bijection
$\eul(X,Y)\xrightarrow{f_*} \eul(N,\partial_0 N)$.
We now define
\[ \tau(N,\partial_0N,\varphi,e,\w):=\tau(X,Y,\varphi\circ f_*,f_*^{-1}(e),f_*^{-1}(\w)).\]
By Lemma  \ref{lem:subdiv} and the discussion in \cite{Tu90} the invariant $\tau(N,\partial_0N,\varphi,e,\w)\in \F$ is well-defined, i.e. independent of the choice of the triangulation.

\section{Duality for Euler structures and homology orientations}
\label{sec:duality for Euler structures}

\subsection{Triangulations of manifolds and dual decompositions}\label{section:dualdecomposition}

In this section we follow closely the notation and language of Turaev in \cite[Section~14]{Tu01}.

Let $X$ be a triangulated space. Given a simplex $a$ we denote by $|a|$ its dimension and we denote by $\ul{a}$ its barycenter.
If $a$ is a face of $b$, then we write $a\leq b$ and we write $a<b$ if $a\leq b$ and $a\ne b$.
If $a_0<a_1<\dots <a_k$ is a sequence of simplices, then we denote by
\[  \ll \ul{a}_0,\ul{a}_1,\dots,\ul{a}_k\rr \]
the convex hull of the points $\ul{a}_0,\ul{a}_1,\dots,\ul{a}_k$. Note that this is a $k$-dimensional simplex contained in $a_k$.
Also note that given a simplex $a$ all simplices of the form $\ll \ul{a}_0,\ul{a}_1,\dots,\ul{a}_k\rr$
with $a_k\leq a$ form a triangulation of $a$, which is called the \emph{first barycentric subdivision of $a$}. If we apply this procedure to all simplices of $X$ we obtain a new triangulation of $X$ which is called the \emph{first barycentric subdivision of $X$}.

Now let $X$ be a triangulation of an $m$-manifold $N$ with (possibly empty)  boundary.
Note that the simplices of $X$ which lie in $\partial N$ form a triangulation of $\partial N$, which we denote by $\partial X$.

Given an $n$-dimensional simplex $a$ of $X$ we define
\[ a^\d:=\bigcup_{a=a_0<a_1<\dots<a_k,k\geq 0} \ll \ul{a}_0,\ul{a}_1,\dots,\ul{a}_k\rr \]
to be the \emph{dual cell of $a$}. Note that $a^\d$ is an $(m-n)$-dimensional cell.
If $a$ is a cell in $\partial X$, then
 we also define
\[ a^\d_{\partial}:=\bigcup_{a=a_0<a_1<\dots<a_k\subset \partial X,k\geq 0} \ll \ul{a}_0,\ul{a}_1,\dots,\ul{a}_k\rr. \]

The \emph{dual cellular decomposition $X^\d$} of $N$ is defined to be the decomposition given by:
\[ \{ a^\d \,|\, a\subset X\} \cup \{a_\partial^\d \,|\, a\subset \partial X\}.\]

\subsection{Dual Euler structures: Euler chains}\label{section:eulerchains}

Let $N$ be a compact $m$-manifold.  We assume that $\chi(N)=\chi(N,\partial N)=0$.
Let $X$ be a triangulation for $N$. Recall that in our notation $Y$ corresponds to $\partial N$.
We denote by $A$ the set of simplices of $X$.
Following \cite[Appendix~B]{Tu90} (see also \cite{HT72}) we define
\[ F:=\sum_{a_0,a_1\in A\atop a_0<a_1} (-1)^{|a_0|+|a_1|}  \ll \ul{a}_0,\ul{a}_1\rr.\]
(Here we orient $\ll \ul{a}_0,\ul{a}_1\rr$ such that $\partial \ll \ul{a}_0,\ul{a}_1\rr=\ul{a}_1-\ul{a}_0$.)
Note that $F$ is a simplicial chain in the first barycentric subdivision of $X$.
It is straightforward to verify, that if $m$ is even, then
\[ \partial F= \sum_{a\subset Y}(-1)^{|a|}\ul{a}\]
and if $m$ is odd, then
\[ \partial F=2\sum_{a\subset X\sm Y}(-1)^{|a|} \ul{a}\,\,+\sum_{a\subset Y}(-1)^{|a|} \ul{a}.\]
(See \cite{HT72} for details.)
We now consider the following map:
\[ \ba{rcl} J \colon \eul(X,Y)&\to & \eul(X) \\[1mm]
[\xi] &\mapsto &[F+(-1)^m \xi].\ea \]
Note that the map $J$ is well-defined and satisfies $J(he)=h^{(-1)^m}J(e)$ for all $h\in \HH_1(N)$ and $e\in \eul(X,Y)$.

The following lemma is an easy generalization of the closed case
(see \cite[Lemma~B.2.1]{Tu90}) which shows that the map $J\colon \eul(N)\to \eul(N)$ is well-defined for closed $N$, not depending on the choice of triangulation of $N$.

\begin{lemma}
Let $X'$ be a subdivision of $X$, then the following diagram commutes:
\[ \xymatrix{ \eul(X,Y)  \ar[d]^\sigma \ar[r]^J & \eul(X) \ar[d]^\sigma \\ \eul(X',Y') \ar[r]^J & \eul(X').}\]
(Here $\s$ denotes the map defined in Section \ref{section:eulersubdivision}.)
\end{lemma}

\begin{proof}
Each open simplex $b$ of $X'$ lies inside a unique simplex $a(b)$ of $X$.
Denote by $\eta$ and $\eta'$ the chains
\[ \sum_{b \subset X' \sm Y'} (-1)^{|b|} [\ul{b}, \ul{a(b)}] \mbox{ and } \sum_{b \subset Y'} (-1)^{|b|} [\ul{b}, \ul{a(b)}] \]
respectively, where $[\ul{b}, \ul{a(b)}]$ is a segment in $a(b)$ starting at $\ul{b}$ and ending at $\ul{a(b)}$.
If $\xi$ is an Euler chain in $(X, Y)$ with $\partial \xi = \sum_{a \subset X \sm Y} (-1)^{|a|} \ul{a}$, then $\sigma([\xi]) = [\xi - \eta]$ (see \cite[Section~1.3]{Tu90}).
Hence
\[ J \circ \sigma([\xi]) = [F' + (-1)^m (\xi - \eta)], \]
where
\[ F' := \sum_{b_0, b_1 \subset X' \atop b_0 < b_1} (-1)^{|b_0| + |b_1|} \ll \ul{b}_0, \ul{b}_1 \rr. \]
Similarly,
\[ \sigma \circ J([\xi]) = [F + (-1)^m \xi - (\eta + \eta')]. \]
Now it suffices to show that the chain $F'$ is homologous to $F + ((-1)^m - 1) \eta - \eta'$.
It is clear that the simplex $\ll \ul{b}_0, \ul{b}_1 \rr$ is homologous to the chain
\[ [\ul{b}_0, \ul{a(b_0)}] - [\ul{b}_1, \ul{a(b_1)}] + \ll  \ul{a(b_0)}, \ul{a(b_1)} \rr. \]
Hence $F'$ is homologous to
\[ \begin{split}
\sum_{b \subset X'} [\ul{b}, \ul{a(b)}] \cdot \left\{ \sum_{b < b'} (-1)^{|b| + |b'|} - \sum_{b' < b} (-1)^{|b| + |b'|} \right\} \\
+ \sum_{a_0, a_1 \subset X \atop a_0 < a_1} \ll \ul{a}_0, \ul{a}_1 \rr \cdot \left\{ \sum_{b_0, b_1 \subset X', b_0 < b_1 \atop a(b_0) = a_0, a(b_1) = a_1} (-1)^{|b_0| + |b_1|} \right\}.
\end{split}  \]
The expression inside the first pair of braces equals
\[ \begin{split}
\sum_{b \subset X'} (-\chi(\mathrm{Link}(b)) - (-1)^{|b|} \chi(\partial b)) = &\sum_{b \subset X' \sm Y'} (- (1 - (-1)^{m - |b|}) - (-1)^{|b|}(1 - (-1)^{|b|})) \\
&+ \sum_{b \subset Y'} (-1 - (-1)^{|b|}(1 - (-1)^{|b|})) \\
= &\sum_{b \subset X' \sm Y'} (-1)^{|b|}((-1)^m - 1) - \sum_{b \subset Y'} (-1)^{|b|}.
\end{split} \]
The sum inside the second pair of braces equals
\[ \begin{split}
\sum_{a(b_0) = a_0} \sum_{b_0 < b_1, a(b_1) = a_1} (-1)^{|b_0|+|b_1|} &= \sum_{a(b_0) = a_0} (-\chi(Z, \partial Z)) \\
&= \sum_{a(b_0) = a_0} (-1)^{|a_1| - |b_0|} = (-1)^{|a_0| + |a_1|},
\end{split}
\]
where $Z$ is the link of $\bar{b}_0$ in $a_1$ and is homeomorphic to $D^{|a_1| - |b_0| - 1}$.
From these we deduce the lemma.
\end{proof}

\subsection{Dual Euler structures: Euler lifts}

We continue with the notation of the previous section.
In particular let $X$ be a triangulation of  $N$. We denote by $Y$ the subcomplex corresponding to $\partial N$.

We pick $e\in \eul(X,Y)=\eul(N,\partial N)$. We denote by $\wti{N}$ and $\wti{X}$ the universal covers. Note that
$\wti{X}$ is a triangulation for $\wti{N}$.
We now pick an Euler lift which represents $e$. Note that each cell $c$ in $\wti{X}\sm \wti{Y}$ intersects precisely one cell $c^\d$ in the universal cover $\wti{X^\d}$ of $X^\d$, which we refer to as the dual cell of $c$ (see Section~\ref{section:dualdecomposition}).
The set of dual cells of the cells in the Euler lift defines an element in $\eul(X^\d)$. Note that the barycentric subdivision of $X$ is also a subdivision of $X^\d$, we thus have a canonical identification map $\eul(X^\d)=\eul(N)$. Following \cite[Appendix~B.2.2]{Tu90} we refer to the resulting  map $\eul(N,\partial N)\to \eul(N)$ as $J'$. Note that for $h\in \HH_1(X)=\HH_1(N)=\HH_1(X^\d)$ we have
\[ J'(he)=h^{(-1)^m}J'(e).\]

The following is an easy  generalization of the closed case
(see \cite[Lemma~B.2.3]{Tu90}).
\begin{lemma}\label{lem:jequal}
The maps $J,J' \colon \eul(N,\partial N)\to \eul(N)$ are identical.
\end{lemma}

\begin{proof}
Fix a point $x \in \wti{N}$.
For each point $y \in \wti{N}$ fix a path in $\wti{N}$ from $x$ and $y$.
Denote by $[x, y]$ the composition of this path and the projection $p \colon \wti{N} \to N$.
So, $[x, y]$ is a path in $N$ going from $p(x)$ to $p(y)$.
Denote by $X'$ the first barycentric subdivision of $X$, and by $A'$ the set of simplices of $X'$.

Let $\{ c_{ij} \}$ be an Euler lift of $(X, Y)$.
The corresponding Euler structure $e \in \eul(X, Y)$ is represented by the Euler chain $\sum_{i = 0}^m \sum_j (-1)^i [x, \ul{c}_{ij}]$.
The Euler structure $J(e) \in \eul(X)$ is represented by the chain
\[ \xi = F + (-1)^m \sum_{i = 0}^m \sum_j (-1)^i [x, \ul{c}_{ij}]. \]
On the other hand, the Euler structure $J'(e) \in \eul(X^{\d})$ is defined by the Euler lift $\{ c_{ij}^{\d} \} \cup \{ d_{ij}^{\d} \} \cup \{ (d_{ij})_{\partial}^{\d} \}$ for an Euler lift $\{ d_{ij} \}$ of $Y$, and is represented by the chain
\[ \xi' = \sum_{i = 0}^{m} \sum_{j} (-1)^{m-i} [x, \ul{c}_{ij}] + \sum_{i = 0}^{m-1} \sum_j ((-1)^{m-i} [x, y_{ij}] + (-1)^{m-i-1} [x, \ul{d}_{ij}]), \]
for some points $y_{ij} \in \mathrm{Int} d_{ij}^{\d}$.
It suffices to show that the images of $[\xi]$ and $[\xi']$ in $\eul(X')$ under the map $\sigma$ defined in Section \ref{section:eulersubdivision} coincide.

Denote by $I$ the set of simplices $c$ in $A'$ such that $c = \ll \ul{a}_0,\ul{a}_1,\dots,\ul{a}_k \rr$ with $a_0 < a_1 < \dots < a_k \in A$ where $a_0\in B$ and $a_k\notin B$. If $c = \ll \ul{a}_0,\ul{a}_1,\dots,\ul{a}_k \rr$ with $a_0 < a_1 < \dots < a_k \in A$, then we define $a(c) = a_0$ and $b(c) = a_k$.

Then the images of $[\xi]$ and $[\xi']$ in $\eul(X')$ are represented by the Euler chains
\[ \mu = \xi + \sum_{c \in A'} (-1)^{|c|} \ll \ul{b(c)}, \ul{c} \rr \mbox{ and } \mu' = \xi' + \sum_{c\in I} (-1)^{|c|}[y_{ij}(c), \ul{d}_{ij}(c)] + \sum_{c \in A'} (-1)^{|c|} \ll \ul{a(c)}, \ul{c} \rr, \]
respectively, where $d_{ij}(c)$ is the $d_{ij}$ such that $p(d_{ij}) = a(c)$ and for each $i,j$, $[y_{ij}, \ul{d}_{ij}]$ is the composition of the inverse path of $[x, y_{ij}]$ and $[x, \ul{d}_{ij}]$.

We have
\[\sum_{c\in I} (-1)^{|c|}[y_{ij}(c), \ul{d}_{ij}(c)] = \sum_{i = 0}^{m-1} \sum_j \left\{\sum_{c\in I\atop a(c) = p(d_{ij})} (-1)^{|c|}\right\}[y_{ij}, \ul{d}_{ij}],\]
and one can see that the sum in the braces is equal to
\[-(\chi(\mathrm{Link}(p(d_{ij}))) - \chi(\mathrm{Link}_\partial(p(d_{ij})))) = -(1-(1+(-1)^{m-1-(i-1)})) = (-1)^{m-i}\]
where $\mathrm{Link}(p(d_{ij}))$ and $\mathrm{Link}_\partial(p(d_{ij}))$ denote the links of $p(d_{ij})$ in $X'$ and in $Y'$, respectively. Therefore we obtain that
\[\mu' = \xi' + \sum_{i = 0}^{m-1} \sum_j (-1)^{m-i} [y_{ij}, \ul{d}_{ij}] + \sum_{c \in A'} (-1)^{|c|} \ll \ul{a(c)}, \ul{c} \rr. \]

Now the chain $\mu - \mu'$ is homologous to
\[ \begin{split}
&F + \sum_{c \in A'} (-1)^{|c|} \ll \ul{b(c)}, \ul{a(c)} \rr \\
= &\sum_{a, b \in A \atop a < b} \ll \ul{a}, \ul{b} \rr \cdot \left[ (-1)^{|a| + |b|} - \left\{ \sum_{c \in A' \atop a(c) = a, b(c) = b} (-1)^{|c|} \right\} \right].
\end{split} \]
The sum in the braces is equal to the Euler characteristic of the simplicial complex whose $k$-dimensional simplices are the sequences $a = a_0 < a_1 < \dots < a_k = b$.
If $d$ is a simplex spanned by the vertices of the simplex $b$ not in $a$, then the $k$-dimensional simplices are in bijective correspondence with the $(k-1)$-dimensional simplices of the first barycentric subdivision of the simplex containing $\ul{d}$ as a vertex.
So the sum is equal to
\[ - \chi(d, \partial d) = - (-1)^{|d|} = (-1)^{|a| + |b|}. \]
It follows that the chain $\mu - \mu'$ is homologous to $0$.
This concludes the proof of the lemma.

\end{proof}

\subsection{Dual Euler structures: combinatorial and smooth structures}

Note that there is a canonical map $J \colon \vect(N,\partial N)\to \vect(N)$  given by inverting the direction of a representative vector field.
The following is  a generalization of \cite[Lemma~B.4]{Tu90}, which shows the lemma when $N$ is closed.

\begin{lemma}\label{lem:canmaps}
Let $N$ be a compact $m$-manifold. Suppose that $\chi(N) = \chi(N,\partial N)  = 0$.
Then the following diagram commutes:
\[ \xymatrix{ \eul(N,\partial N)  \ar[d]^{ca_N}\ar[r]^-J & \eul(N) \ar[d]^{ca_N} \\ \vect(N,\partial N)\ar[r]^-J & \vect(N).}\]
\end{lemma}

\begin{proof}
The case when $\partial N = \emptyset$ is proved in \cite{Tu90}, and therefore we assume that $\partial N$ is not empty. First, for the reader's convenience we recall the definitions of the maps $ca_N\colon \eul(N,\partial N) \to \vect(N,\partial N)$ and $ca_N\colon \eul(N) \to \vect(N)$ in \cite[Section~6.6]{Tu90}.

Let $(X,Y,t)$ with $t\colon X\to N$  a triangulation of $(N,\partial N)$. Recall that  $Y$ denotes the subcomplex of $X$ consisting of the simplices in $\partial N$. Define $N'$ to be the manifold obtained by attaching $N$ to $\partial N \times [0,1]$ via the natural homeomorphism $\partial N \to \partial N\times 0$. We extend the triangulation $Y \times 0 \coprod Y \times 1$ of the manifold $\partial N\times\{0,1\}$ to some triangulation $k$ of the manifold $\partial N \times[0,1]$ without adding new vertices.
Then we obtain a triangulation of $N'$ using $X$ and $k$, which we denote by $X\cup  k$.
 Let $I$ be the set of the simplices of $k$ that do not lie in $Y\times 0\coprod Y\times 1$. For a simplex $a\notin I$, let $\ul{a}$ be the barycenter of $a$ as usual but if $a\in I$, then let $\ul{a}$ be an (arbitrary) interior point of the simplex $a$ lying in $\partial N\times (1/2)$.
Then we obtain a canonical singular vector field $F_1$ on $N'$ defined in terms of the points $\ul{a}$ where $a\in X\cup k$ (see \cite[Section~6.3]{Tu90}). The singular points of the vector field $F_1$ are exactly the points $\ul{a}$ with $a\in X\cup k$. Moreover, the field $F_1$ is transversal to $\partial N\times (3/4)$ and is directed towards $\partial N\times 0$. Also it is transversal to $\partial N\times (1/4)$ and is directed towards $\partial N\times 1$.

Now we define $ca_N\colon \eul(N,\partial N) \to \vect(N, \partial N)$. Let $\pr_1\colon |k|\to \partial N$ be the projection. We may assume that $\pr_1$ linearly maps simplices onto simplices. For each $a\in I$ denote by $\g_a$ some path $[0,1]\to \pr_1(a)\times [0,1/2]$ from $\ul{a}$ to the barycenter of the simplex $\pr_1(a)$. Let $\G:=-\sum_{a\in I} (-1)^{|a|}\g_a$.

Let $\xi\in \eul(N,\partial N)$. Let $M:=N'\setminus \{\partial N\times (3/4,1])$. Then $\partial M = \partial N\times (3/4)$, and $F_1|_M$ is transversal to $\partial M$ and is directed inwards on $\partial M$. Moreover, one can check that $\partial(\G +\xi) = \sum_{b\in X\cup I} (-1)^{|b|}\ul{b}$. Therefore $\ul{b}$ runs over all singular points of $F_1$ in $M$. The chain $\G+\xi$ can be presented by a chain consisting of oriented arcs joining a point of $X$ to all $\ul{b}$ where $b\in X\cup I$. Note that the vector field $F_1|_M$ is non-singular outside a neighborhood of these arcs. Since $\chi(M)=\chi(N) = 0$, this non-singular vector field defined on the complement of a neighborhood of the arcs in $M$ can be extended to a non-singular vector field $\bar{F}_\xi$ on $(M,\partial M)$.
The vector field $\bar{F}_\xi$ is transferred to $N$ by the canonical diffeomorphism $g\colon M\to N$. Finally, $ca_N(\xi)$ for $\xi\in \eul(N,\partial N)$ is defined to be the homology class of this vector field, $[dg(\bar{F}_\xi)]\in \vect(N,\partial N)$.

On the other hand, for $\xi\in \eul(N)$, $ca_N(\xi)\colon \eul(N) \to \vect(N)$ is defined as follows. Let $L:=N'\setminus \{\partial N\times (1/4,1])\}$ and $F_1$ be the vector field on $N'$ defined as above. Then $\partial L = N\times (1/4)$ and $F_1|_{L}$ is transversal to $\partial L$ and is directed outwards on $\partial L$. Moreover, $\partial \xi = \sum_{b\in X} (-1)^{|b|}\ul{b}$ and therefore $\ul{b}$ runs over all singular points of $F_1$ in $L$. Now the chain $\xi$ can be presented by a chain consisting of oriented arcs joining a point of $X$ to all $\ul{b}$ where $b\in X$ and the vector field $F_1|_{L}$ is non-singular outside a neighborhood of these arcs. Since $\chi(L)=\chi(N) = 0$, similar to the case of $\eul(N,\partial N)$, we can obtain a non-singular vector field $\bar{F}_\xi$ on $L$. This vector field is transferred to $N$ by the
canonical diffeomorphism $h\colon L \to N$ and finally $ca_N(\xi)$ for $\xi\in \eul(N)$ is defined to be the homology class  $[dh(\bar{F}_\xi)]\in \vect(N)$.

Now we prove that the diagram commutes. Let $\xi\in \eul(N,\partial N)$. Let $g'\colon M \to L$ be a diffeomorphism constructed in a similar fashion as $g\colon M\to N$ is defined such that $\partial N\times [1/2,3/4]\subset M$ is sent to $\partial N\times [0,1/4]\subset L$. Also we may assume that $h\circ g' = g$. To show the commutativity of the diagram, it suffices to show that $J(dg'(\bar{F}_\xi))$ is homologous to  $\bar{F}_{J(\xi)}$ in $L$. Note that the singular points of the vector fields $J(dg'(F_1|_M))$ and $F_1|_L$ lie in $N\subset L$, and therefore one can show that the restrictions of $J(dg'(\bar{F}_\xi))$ and $\bar{F}_{J(\xi)}$ to $N$ are homologous in $N$ using the arguments for the closed case in \cite[Lemma B.4]{Tu90}. In fact, their restrictions to $\partial N\times 0$ are homotopic.  Let $L' := \overline{L\setminus N} = \partial N\times [0,1/4]$. Now it is enough to show that $J(dg'(\bar{F}_\xi))$ and $\bar{F}_{J(\xi)}$ are homotopic on $(L',\partial L')$. The obstruction to the homotopy of $u:=J(dg'(\bar{F}_\xi))|_{L'}$ and $v:=\bar{F}_{J(\xi)}|_{L'}$ on $(L',\partial L')$ is an element of the group $H^{m-1}(L',\partial L';\Z^\omega)\cong \HH_1(L')$ where $\omega \colon \pi_1 L' \to \aut(\Z)$ is the first Stiefel-Whitney class of the manifold $L'$ (see \cite[Section 5.2]{Tu90}).
Let $\ell\in \HH_1(L')$ be that element.
But since $H_1(L';\Z)\cong H_1(\partial N\times 0;\Z)$ and $u|_{\partial N\times 0}$ and $v|_{\partial N\times 0}$ are homotopic, one can see that $\ell = 1$ and hence $u$ is homotopic to $v$ in $L'$.
\end{proof}

\subsection{Summary and dual maps}

Let $N$ be a compact $m$-manifold with $\chi(N) = \chi(N,\partial N)  = 0$. Let $(X,Y,t)$ with $t\colon X\to N$ be a triangulation of $(N,\partial N)$.
We denote by $F$ the chain introduced in Section \ref{section:eulerchains}.
In the previous sections we showed that the maps
\[ \ba{rcl} \eul(X,Y)&\xrightarrow{J} & \eul(X) \\[1mm]
[\xi]&\mapsto & [F+(-1)^m\xi]\ea \mbox{ and } \ba{rcl} \eul(X,Y)&\xrightarrow{J'} & \eul(X^\dagger) \\[1mm]
[c]&\mapsto & [c^\dagger]\ea \]
give rise to well--defined maps $J,J':\eul(N,\partial N)\to \eul(N)$ and we saw that the maps are in fact the same.
We furthermore studied the  $J \colon \vect(N,\partial N)\to \vect(N)$  given by inverting the direction of a representative vector field and we showed  that the following diagram commutes:
\[ \xymatrix{ \eul(N,\partial N)  \ar[d]^{ca_N}\ar[r]^-J & \eul(N) \ar[d]^{ca_N} \\ \vect(N,\partial N)\ar[r]^-J & \vect(N).}\]

We now denote the inverse maps $\eul(N)\to \eul(N,\partial N)$ and $\vect(N)\to \vect(N,\partial N)$ by $J$ as well.
It is clear that $J:\vect(N)\to \vect(N,\partial N)$ is given by inverting the direction of a representative vector field
and it is clear that $\eul(N)\to \eul(N,\partial N)$ corresponds to the map
\[ \ba{rcl} \eul(X)&\xrightarrow{J} & \eul(X,Y) \\[1mm]
[\xi]&\mapsto & [(-1)^{m+1}F+(-1)^m\xi].\ea \]
In the following, given $e\in \vect(N)=\eul(N)$ or $e\in \vect(N,\partial N)=\eul(N,\partial N)$
we write $e^\d=J(e)$. Note that by definition we have $(e^\d)^\d=e$.

\subsection{Dual homology orientations}\label{section:dualhomol}

Let $N$ be an oriented $m$-manifold and
let $\partial_0 N$ be the empty set or $\partial N$. We denote by $[N]\in H_m(N,\partial N;\Z)$ the orientation class and denote by $D\colon H_i(N,\partial_0 N;\R)\to H^{m-i}(N,\partial N\sm \partial_0 N;\R)$ the isomorphism given by Poincar\'e duality.
Following \cite[Section~VI.11]{Br93} we then define  the intersection pairings on $N$ as follows:
\[ \ba{ccl} H_i(N,\partial N;\R)\times H_{m-i}(N;\R)&\to& \R \\
(a,b)&\mapsto  & a\cdot b:= (D(b)\cup D(a))\cap [N].\ea \]
Note that the intersection pairings are non-degenerate.

Let $\w$ be a homology orientation of $(N,\partial N)$.
We pick an  orientation for $N$. Following  \cite[p.~178]{Tu86} we will now define a dual homology orientation for $N$.
First note that the above  pairings
 give rise to a non-degenerate pairing
\[ H_*(N,\partial N;\R)\times H_*(N;\R)\to \R.\]
An ordered basis $w_i$ for $H_*(N;\R)$ is called positive if for one, and hence any, positive basis $v_i$ for $H_*(N,\partial N;\R)$
the determinant of the matrix $v_i\cdot w_j$ is positive.
The resulting homology orientation on $N$ is called the \emph{homology orientation dual to $\w$} and denoted by $\w^\d$.
Note that if $\chi(N)=0$, and this is the only case we will consider in this paper, then it follows from $\dim H_*(N;\R)\equiv \chi(N)\mbox{ mod 2}$
that  the definition of the dual homology orientation does not depend on the choice of $[N]$. \
On the other hand, if $\w$ is a homology orientation on $N$, then the homology orientation on $(N,\partial N)$ obtained via the above non-degenerate pairing  is also called the \emph{homology orientation dual to $\w$} and denoted by $\w^\d$. Note that if $\w$ is a homology orientation on $N$ or $(N,\partial N)$ then $(\w^\d)^\d = \w$.

Let $N$ be a closed $m$-manifold. Following Turaev \cite[Appendix,~Theorem~5]{Tu86} we now define
\[ z(N):=\left\{ \ba{ll} 0, &\mbox{ if }m\equiv 2 \mod 4 \mbox{ or }m\equiv 3 \mod 4, \\
\sum_{i=0}^{[m/2]} b_i(N), &\mbox{ if }m\equiv 1\mod 4,\\
\sum_{i=0}^{m/2} b_i(N)\,\,+\frac{1}{2}(b_{\frac{m}{2}}(N)-\sign(N)), &\mbox{ if }m\equiv 0\mod 4.\ea \right.\]
The following theorem is due to Turaev (see \cite[p.~179]{Tu86}):

\begin{theorem}\label{thm:dualor}
Let $N$ be a closed $m$-manifold and $\w$ a homology orientation. Then
\[\w^\d = (-1)^{n}\w \]
where
\[ n=z(N)+\sum_{i=0}^m\b_i(N)\b_{i-1}(N)\,\,+\sum_{i=0}^{[m/2]}\b_{2i}(N).\]
\end{theorem}

\section{Duality for torsion of manifolds equipped with Euler structures}
\label{sec:duality for torsion}

\subsection{The algebraic duality theorem for torsion}

Let $\F$ be a field with involution. Given a vector space $V$ we denote by $\ol{V}$ the vector space with the same underlying abelian group but with involuted $\F$-structure.

In the following let $C_*$ be a chain complex of length $m$ over $\F$.
Suppose that $C_*$ is equipped with an ordered basis $c_i$ for each $C_i$. We write $H_i=H_i(C)$
and we suppose that $H_*$ is equipped with an ordered basis $h_*$.
We denote by $C^\d$ the \emph{dual chain complex}
as defined in \cite[Section~2.2.2]{Tu86}: the chain groups
are $C^\d_i:=\ol{\hom_\F(C_{m-i},\F)}$ and $\partial_i:C_{i+1}^\d\to C_i^\d$
is given by $(-1)^{m-i}\partial_{m-i-1}^*$. We denote by $c_*^\d$ and $h_*^\d$ the bases of $C^\d$ respectively $H_*(C^\d)$ dual to the bases $c_*$ and $h_*$.

Following \cite[p.~142]{Tu86} we define
\[ \ba{rcl}
 r(C_*)&=&\sum_{i=0}^m \big(\a_i(C)\a_{i-1}(C)+\b_i(C)\b_{i-1}(C)\big)+\sum_{i=0}^{[m/2]}\big(\a_{2i}(C)+\b_{2i}(C)\big).\ea \]
When $C_*$ is the ordinary chain complex of a pair of spaces $(X,Y)$, then we write $r(X,Y)=r(C_*(X,Y))$.
The following lemma is now an immediate consequence of Lemma 7 in the appendix of \cite{Tu86}:

\begin{lemma}\label{lem:dual}
We have
\[ r(C_*)=r(C_*^\d)\]
and the following equality holds:
\[\tau(C_*,c_*,h_*) =(-1)^{r(C_*)}\tau(C_*^\d,c_*^\d,h_*^\d)^{(-1)^{m+1}}.\]
\end{lemma}

\subsection{The duality theorem for manifolds}

In this section we will prove the following duality theorem.

\begin{theorem}\label{thm:dualitygeneral}
Let $N$ be a compact, orientable $m$-manifold.  We assume that $\chi(N)=\chi(N,\partial N)=0$.
Let $\varphi \colon \pi_1(N)\to \gl(d,\F)$ be a representation over a field with involution, let $e$ be  an Euler structure for $(N,\partial N)$ and let $\w$ be a homology orientation
for $(N,\partial N)$. Then
 $H_*^\varphi(N,\partial N;\F^d)=0$ if and only if $H_{*}^{\varphi^\d}(N;\F^d)=0$.
 Furthermore, if these groups vanish,
 then
\[ \tau(N,\partial N,\varphi,e,w)=(-1)^{ds(N)}\ol{\tau(N,\varphi^\d,e^\d,\w^\d)^{(-1)^{m+1}}},\]
where
$\varphi^\d$ is the dual of $\varphi$ and
\[  s(N)=\sum_{i=0}^m \b_i(N)\b_{i-1}(N)+\sum_{i=0}^{[m/2]}\b_{2i}(N). \]
\end{theorem}

The proof of this theorem will require the remainder of this subsection.

Let $X$ be a triangulation for $N$. Let $X^\d$ be the  CW complex dual to $(X,Y)$ of zero Euler characteristic.
 We write $\pi=\pi_1(X)=\pi_1(N)=\pi_1(X^\d)$.
Let $\varphi \colon \pi\to \gl(d,\F)$ be a representation over a field with involution.
First note that by Poincar\'e duality and the universal coefficient theorem we have
\[ H_i^\varphi(N,\partial N;\F^d)\cong H^{m-i}_{\varphi^\d}(N;\F^d)\cong H_{m-i}^{\varphi^\d}(N;\F^d).\]
In particular $H_*^\varphi(N,\partial N;\F^d)= 0$ if and only $H_{*}^{\varphi^\d}(N;\F^d)=0$.
Now suppose that these homology groups are indeed zero.

First note that $X^\d$ and $X$ share a common subdivision, namely the first barycentric subdivision. By Lemma \ref{lem:subdiv}
we can therefore use $X^\d$ to calculate the twisted torsion of $N$.

For the remainder of this section we pick an Euler lift which represents $e\in \eul(X,Y)$.
We pick an ordering for the cells and for each cell we pick an orientation. We denote the resulting oriented $i$-cells by $c_{i1},\dots,c_{is_i}$.
We now also pick an orientation for $N$. For each $(i,j)$ we now denote by $c^\d_{ij}$ the unique oriented cell in $\wti{X^\d}$ which has intersection number $+1$ with $c_{ij}$.
We denote by $\ol{c_{ij}}$ and $\ol{c^\d_{ij}}$ the projection of the cells to $X$ respectively $X^\d$.
Let $h_*$ be a basis of $H_*(N,\partial N;\R)=H_*(X,Y;\R)$ which represents the homology orientation $\w$ and denote by $h_*^\d$ the dual basis of $H_*(N;\R)=H_*(X^\d;\R)$.

Note that by Lemma \ref{lem:jequal} the Euler lift $\{c^\d_{ij}\}$ represents $e^\d = J(e) \in \eul(N)$.
Theorem \ref{thm:dualitygeneral} now follows immediately from the definitions, the equalities $ \a_i^\varphi(X)=d\a_i(X)$ and the following two lemmas.

\begin{lemma}
The following equality holds:
\[ \tau(C_*(X,Y),c_*,h_*)=(-1)^r\tau(X^\d,\{c_{ij}^\d\},h_*^\d)^{(-1)^{m+1}},\]
where
\[ r:=\sum_{i=0}^m \big(\a_i(X)\a_{i-1}(X)+\b_i(X)\b_{i-1}(X)\big)+\sum_{i=0}^{[m/2]}\big(\a_{2i}(X)+\b_{2i}(X)\big).\]
Furthermore
\[ \eta(X,Y)=\eta(X^\d).\]
\end{lemma}

\begin{proof}
For the equality $\eta(X,Y)=\eta(X^\d)$ we refer to proof of Lemma 7 in the appendix of \cite{Tu86}.
Note that for each $i$ there is a canonical, non-singular intersection pairing
\[C_{m-i}({X},Y;\R)\times   C_{i}({X^\d};\R)  \to \R\]
which is induced by the condition $a\cdot b^\d=\delta_{ab}$ for cells of $X\sm Y$ and $X^\d$.
By  \cite[Claim~14.4]{Tu01} we then obtain  the following  commutative diagram:
\[\ba{cclcl}
C_{i+1}({X,Y};\R)
&\times& C_{m-i-1}({X^\d};\R)
&\to &\R\\[2mm]
\downarrow \partial_i &&\hspace{0.7cm} \uparrow (-1)^{i+1}\partial_{m-i-1} &&\downarrow = \\[2mm]
 C_{i}({X,Y};\R) &\times &C_{m-i}({X^\d};\R) &\to &\R.\ea \]
This shows that we can identify $C_*(X^\d;\R)$ with the chain complex dual to $C_*(X,Y;\R)$.
It is straightforward to verify that under this identification the bases  $c_*^{\d}$ and $h_*^\d$ are dual to the bases of $c_*$ and $h_*$.
The lemma now follows immediately from Lemma \ref{lem:dual}.
\end{proof}

\begin{lemma}
\[ \tau(C_*^\varphi(X,Y;\F^d),\{c_{ij}\otimes v_k\})=(-1)^{r^\varphi}\ol{\tau(C_*^{\varphi^\d}(X^\d;\F^d),\{c_{ij}^\d\otimes v_k\})^{(-1)^{m+1}}}.\]
where
\[ r^\varphi:=\sum_{i=0}^m \a^\varphi_i(X)\a^\varphi_{i-1}(X)+\sum_{i=0}^{[m/2]}\a^\varphi_{2i}(X).\]
\end{lemma}

\begin{proof}
Similar to the proof of the previous lemma  there is a canonical, non-singular intersection pairing
\[  C_{m-i}(\wti{X},\wti{Y})  \times  C_{i}(\wti{X^\d})\to \Z\]
which is induced by the condition $a\cdot b^\d=\delta_{ab}$ for cells of $\wti{X}\sm \wti{Y}$ and of $\wti{X^\d}$. We can turn this pairing into a sesquilinear, non-singular pairing over $\Z[\pi]$:
\[\ba{ccl} C_{m-i}(\wti{X},\wti{Y}) \times C_{i}(\wti{X^\d})    &\to &\Z[\pi]\\
(a,b) &\mapsto & \ll a,b\rr :=\sum_{g\in \pi} (a\cdot bg)g^{-1}.\ea \]
Note that for $g,h\in \pi$ we have $\ll ag,bh\rr =g^{-1}\ll a,b\rr h$.
This pairing has the property (see e.g. \cite[Claim~14.4]{Tu01}) that the following diagram commutes:
\[\ba{cclcl} C_{i+1}(\wti{X},\wti{Y}) &\times &C_{m-i-1}(\wti{X^\d})&\to &\Z[\pi]\\[2mm]
\downarrow \partial_i&&\hspace{0.7cm} \uparrow(-1)^{i+1}\partial_{m-i-1}&&\downarrow = \\[2mm]
 C_{i}(\wti{X},\wti{Y})   &\times & C_{m-i}(\wti{X^\d}) & \to &\Z[\pi].\ea \]

We now consider the chain complexes
\[ \ba{rcl} C_*^\varphi(\wti{X},\wti{Y};\F^d)&=&C_{*}(\wti{X},\wti{Y})\otimes_{\Z[\pi]}\F^d\\
C_*^{\varphi^\d}(\wti{X^\d};\F^d)&=&C_{*}(\wti{X^\d})
\otimes_{\Z[\pi]}\F^d.\ea \]
 We consider the following pairing:
\[\ba{ccl} \l:C_{m-i}^\varphi(\wti{X},\wti{Y};\F^d)\times C_{i}^{\varphi^\d}(\wti{X^\d};\F^d)&\to &\F\\
(a\otimes v,b\otimes w) &\mapsto & \ol{v}^t\varphi(\ll a,b\rr )w.\ea \]
(Here we denote by $\varphi$ the ring homomorphism $\Z[\pi]\to M(n,\F)$ induced by $\varphi \colon \pi\to \gl(n,\F)$.)
It is straightforward to verify that this pairing is well-defined and non-singular and that the  bases $\{c_{ij}\otimes v_k\})$ and $\{c_{ij}^\d\otimes v_k\}$ are dual to each other.
It follows from the above that the following diagram also commutes:
\[\ba{cclcl}
  C_{i+1}^\varphi(\wti{X},\wti{Y};\F^d)
&\times&C_{m-i-1}^{\varphi^\d}(\wti{X^\d};\F^d)
&\to &\F\\[2mm]
\downarrow \partial_i &&\hspace{0.7cm}\uparrow (-1)^{i+1}\partial_{m-i-1} &&\downarrow = \\[2mm]
 C_{i}^\varphi(\wti{X},\wti{Y};\F^d)  &\times &C_{m-i}^{\varphi^\d}(\wti{X^\d};\F^d)   &\to &\F.\ea \]
The lemma now follows immediately from Lemma \ref{lem:dual} and from the fact that $\b_i^\varphi(X)=0$ for all $i$.
\end{proof}

\subsection{The duality theorem for closed manifolds}

Let $N$ be a closed, orientable $m$-manifold.
Note that given $e\in \eul(N)$ the dual Euler structure $J(e)$ also lies in $\eul(N)$.
The unique element $g\in \HH_1(N)$ with $e=g\cdot J(e) = g\cdot e^\d$ is called the \emph{Chern class of $e$} and denoted by $c_1(e)$.

 Combining Theorem \ref{thm:dualitygeneral}
with Theorem \ref{thm:dualor} and  Lemma \ref{lem:indet} now immediately gives us the following result.

\begin{theorem}\label{thm:dualitygeneralclosed}
Let $N$ be a closed, orientable $m$-manifold with $\chi(N)=0$.
Let $\varphi \colon \pi_1(N)\to \gl(d,\F)$ be a representation over a field with involution, let $e$ be  an Euler structure for $N$ and let $\w$ be a homology orientation
for $N$. Then
\[ \tau(N,\varphi,e,w)=(-1)^{dz(N)}\cdot \ol{\det(\varphi(c_1(e)))}\cdot \ol{\tau(N,\varphi^\d,e,\w)^{(-1)^{m+1}}}.\]
\end{theorem}

\section{Twisted torsion of 3-manifolds with toroidal boundary}
\label{sec:twisted torsion of 3-manifolds}

\subsection{Torsions of exact sequences of based chain complexes}

Let
\[ 0\to C'_*\xrightarrow{\i} C_*\to C_*''\to 0\]
be a short exact sequence of chain complexes of finite length over the field $\F$.
Suppose that for each $i$ we have a basis $c_i',c_i,c_i''$ for the corresponding chain groups.
We say that the bases are \emph{compatible} if  the basis $c_i$ is of the form $\{\i(c_i'),d_i\}$, where
the ordered set $d_i$ maps to the ordered set $c_i''$ under the projection map.

Furthermore  suppose that we have bases $h_i',h_i,h_i''$ for the corresponding homology groups.
We denote by $\mathcal{H}$ the long exact sequence in homology of the above short exact sequence of chain complexes. Note that the bases $h_i',h_i,h_i''$ turn $\mathcal{H}$ into a based chain complex. We denotes this basing for $\mathcal{H}$ by $\mathcal{B}_*$.

The following lemma is now \cite[Lemma~3.4.2]{Tu86}:

\begin{lemma}\label{lem:ses}
Let
\[ \ba{rcl} \nu&=& \sum_i \a_i(C'')\a_{i-1}(C'), \\
\mu&=&\sum_i \big((\b_i(C)+1)(\b_i(C')+\b_i(C''))+\b_{i-1}(C')\b_i(C'')\big). \ea \]
If $\check{\tau}(C_*',c_*',h_*')\ne 0$ or $\check{\tau}(C_*'',c_*'',h_*'')\ne 0$ and if $c_*',c_*,c_*''$ are compatible bases,
then
\[ \check{\tau}(C_*,c_*,h_*)=(-1)^{\nu+\mu}\check{\tau}(C_*',c_*',h_*')\cdot \check{\tau}(C_*'',c_*'',h_*'')\cdot \tau(\mathcal{H},\mathcal{B}_*).\]
\end{lemma}

\subsection{Canonical  structures on tori}

Let $T$ be a torus.
We denote by $F$ the free abelian group $\pi_1(T)$.  We denote by $\Q(F)$ the quotient field of $\Z[F]$.
We denote by
\[ \varphi \colon \pi_1(T)=F\to \aut(\Z[F]) \to \gl(1,\Q(F))\]
 the regular representation.
It is well-known that for any $e\in \eul(T)$ and any homology orientation $\w$ of $T$ we have that
$\tau(T,\varphi,e,\w)\in \pm F$. (This can for example be seen from an explicit calculation as in the proof of Lemma \ref{lem:canor}.)

It now follows from Lemma \ref{lem:indet} that there exists a unique Euler structure $e$ and a unique homology orientation $\w$ such that
\[ {\tau}(T,\varphi,e,\w)=1.\]
We refer to this  Euler structure as the \emph{canonical Euler structure on $T$} and we refer to $\w$ as the \emph{canonical homology orientation of $T$}.
We refer to a vector field on $T$ representing this Euler structure as a \emph{canonical vector field}.
This definition is identical to the definition provided by  Turaev \cite[p.~10]{Tu02}.

It is straightforward to see that the canonical Euler structure on $T$ is the unique Euler structure which is invariant under the action by $\aut(T)$ on $\eul(T)$.
Also note, that if $e$ is the canonical Euler structure and $\w$ the canonical homology orientation,
then a straightforward calculation shows that for any representation $\varphi \colon \pi_1(T)\to \gl(d,\F)$ such that $H_*^\varphi(T;\F^d)=0$ we have that
\be \label{equ:taue1} {\tau}(T,\varphi,e,\w)=1.\ee

For future reference we also state the following lemma:

\begin{lemma}\label{lem:canor}
Let $T$ be a torus and $\w$ the canonical homology orientation.
We now pick a point $P$ in $T$, two homology classes $x,y$ in $H_1(T;\Z)$ and we pick a generator $A$ of $H_2(T;\Z)$ such that the intersection number $x\cdot y$ is $+1$ with respect to the orientation $A$ of $H_2(T;\Z)$.
We denote by $\w'$ the homology orientation given by $\{P,x,y,A\}$. Then
\[ \w'=-\w.\]
\end{lemma}

\begin{proof}
We pick a homeomorphism  $T\cong S^1\times S^1$
such that $x$ corresponds to $S^1\times *$ with the canonical orientation and such that $y$ corresponds to $*\times S^1$ with the canonical orientation.
We equip $S^1\times S^1$ with the standard CW structure given by $*\times *, S^1\times *, *\times S^1$ and $S^1\times S^1$.
Note that with the canonical orientations these  four cells of the CW decomposition  match the homology orientation $\w'$.
(See e.g. \cite[Example~VI.11.12]{Br93} for details.)
We can canonically identify the quotient field of $\Z[\pi_1(T)]$ with $\Q(x,y)$.
If we pick appropriate lifts of the above cells to the universal abelian cover $\wti{X}$, then the chain complex $C_*(\wti{Z};\Q(x,y))$ is given as follows:
\[ 0\to \Q(x,y)\xrightarrow{\bp y-1\\1-x\ep}\Q(x,y)^2 \xrightarrow{\bp 1-x&1-y\ep}\Q(x,y)\to 0.\]
It now follows from \cite[Theorem~2.2~and~Remark~2.4]{Tu01} that the torsion of this based complex is $\frac{y-1}{1-y}=-1$.
\end{proof}

\subsection{Maps on Euler structures and Chern classes on 3-manifolds with toroidal boundary}

Let $N$ be a compact, orientable 3-manifold with toroidal boundary and let $e\in \eul(N,\partial N)$.

Let $X$ be a triangulation for $N$.  We denote the subcomplexes corresponding to the boundary components of $N$ by $T_1\cup \dots\cup T_b$.
We now denote by $p \colon \wti{X}\to X$ and $p_i \colon \wti{T_i}\to T_i, i=1,\dots,b$ the universal covering maps of $N$ and $T_i,i=1,\dots,b$.
We pick an Euler lift $c$ which represents $e$. For each boundary torus $T_i$ we pick  an Euler lift $t_i$ which represents the canonical Euler structure.

We denote by $\wti{T_i}'$ the cover of $T_i$ corresponding to the kernel of the map $\pi_1(T_i)\to \pi_1(N)$.
Note that  we have induced maps $p^{-1}(T_i)\to \wti{T_i}'$ and $\wti{T_i}\to \wti{T_i}'$.  We denote by $t_i'$ the image of $t_i$ under the map  $\wti{T_i}\to \wti{T_i}'$.
(Note that $\wti{T_i}'=\wti{T_i}$ if the boundary of $N$  is incompressible.)
For each $i$ we now pick a collection $\ti{t}_i$ of cells in $p^{-1}(T_i)$ such that each cell in $t_i'$ is covered by exactly one cell in $\ti{t}_i$.
The set of cells $\{\ti{t}_1,\dots,\ti{t}_b,c\}$ now defines an Euler structure for $N$, which only depends on $e$.
We denote by $K$ the resulting map $\eul(N,\partial N)\to \eul(N)$. Note that this map is $\HH_1(N)$-equivariant.

 We now let $e\in \vect(N,\partial N)$. We represent $e$ by a regular vector field $v$. Recall that $v$ is a nowhere vanishing vector field which points inwards on $\partial N$.
We  recall the following construction of Turaev \cite[p.~11]{Tu02}: For each boundary torus $T_i$ pick  a canonical vector field $u_i$.
We then denote  by $w$ the vector field on $N$ which agrees with $v$ on the complement of $\partial N\times [0,1]$, and it equals $\cos(\pi t)v(x,t)+\sin(t\pi)u_i(x)$ on $T_i\times [0,1]$, where $x\in T_i$ and $t\in [0,1]$.
Note that $w$ now points outward on the boundary and thus defines an element in $\eul(N)$. Note that the homology class of $w$ is well-defined, i.e. it only depends on $e$ and it is independent of the choices involved.
We refer to Turaev  \cite[p.~10]{Tu02} for details. We denote  the resulting map $\vect(N,\partial N)\to \vect(N)$ by $K$. It is straightforward to verify that $K$ is $\HH_1(N)$-equivariant.

We then have the following lemma:

\begin{lemma}
The following diagram commutes:
\[ \xymatrix{ \eul(N,\partial N)\ar[d]^{ca_N}\ar[r]^-K & \eul(N)\ar[d]^{ca_N} \\ \vect(N,\partial N)\ar[r]^-K & \vect(N).}\]
\end{lemma}

This lemma follows from the definitions and careful reading of \cite[Section~6.6]{Tu90}.
We will thus only provide a short outline of the proof.

\begin{proof}
Since the constructions only happen on neighborhoods of the boundary tori it suffices to consider the following situation.
Let $T$ be a torus and $I=[0,1]$. Then the above constructions and the construction of Section  \ref{section:smootheuler} also define maps
\[ \xymatrix{ \eul(T\times I,T\times 1)\ar[d]^\cong \ar[r]^-K &  \eul(T\times I)\ar[d]^\cong \\
\vect(T\times I,T\times 1)\ar[r]^-K & \vect(T\times I).}\]
Note that $\aut(T)$ acts in a canonical way on $T\times I$ and on all of the above objects.
Recall that the canonical Euler structure is invariant under the $\aut(T)$--action.
Also note that  the horizontal maps in the above diagram are  $\aut(T)$--equivariant.
Furthermore note that  the vertical maps are $\aut(T)$--equivariant by the explicit construction
in \cite[Section~6.6]{Tu90}.

It follows that the various maps send the unique structure invariant under the $\aut(T)$--action to the unique structure invariant under the $\aut(T)$--action.
Since all the above objects admit a free and transitive action by $\HH_1(T)$ and since all the above maps are $\HH_1(T)$--equivariant under this action it now follows immediately that the above diagram commutes.

\end{proof}

Given $e\in \eul(N)=\vect(N)$
there exists a unique element $g\in \HH_1(N)$ such that $e=g\cdot K(e^\d)$.
Following Turaev \cite[p.~11]{Tu02} we now define $c_1(e):=g$. We refer to $c_1(e)$ as the \emph{Chern class of $e$}.

\subsection{Torsions of 3-manifolds with toroidal boundary}

\begin{theorem}\label{thm:torsionboundary}
Let $N$ be a  3-manifold with non-empty toroidal boundary, let  $\varphi\colon \pi_1(N)\to \gl(d,\FF)$ be a representation over a field $\FF$, let $e \in \eul(N,\partial N)$
and let $\w\in \ort(N,\partial N)$. Suppose that $H_*^\varphi(\partial N;\F^d)=0$.  Then
\[ \tau(N,\partial N,\varphi,e,\w)=(-1)^{d(b_1(N)+b_0(\partial N))} \tau(N,\varphi,K(e),\w^\d).\]
\end{theorem}

\noindent The proof of Theorem~\ref{thm:torsionboundary} is postponed to the end of the subsection.

It follows from the long exact sequence of the pair $(N,\partial N)$, and our assumption that $H_*^\varphi(\partial N;\F^d)=0$,
that $H_*^\varphi(N;\F^d)=0$ if and only if $H_*^\varphi(N,\partial N;\F^d)=0$. In particular either both torsions are zero, or both are non-zero.
For the remainder of this proof we now assume that $H_*^\varphi(N;\F^d)=0$.

We  pick a triangulation $X$ for $N$. As usual we denote by $Y$ the subcomplex corresponding to $\partial N$.
Let $e\in \eul(N,\partial N)=\eul(X,Y)$. We pick an Euler lift $c_*$ which represents $e$.  We denote the components of $Y$ by $Y_1\cup\dots \cup Y_b$
and we pick Euler lifts $t^1_*,\dots,t^b_*$ which represent the canonical Euler structures. We denote by $\w_1,\dots,\w_b$ the canonical homology orientations of $Y_1,\dots,Y_b$.
Note that $\sign(\check{\tau}(C_*({Y}_i),t^i_*,\w_i)))=\pm 1$. After possibly swapping two cells we can assume that in fact
\be \label{equ:tau1} \sign(\check{\tau}(C_*({Y}_i),t^i_*,\w_i))=1.\ee
(Here and in the remainder of this section we denote by $c_*$ and $t^i_*$ also the projections of the cells to $X$ and $Y_i$.)

 We pick $\ti{t}^1_*,\dots, \ti{t}^b_*$ as in the previous section.
We write $\ti{t}_*=\ti{t}^1_*\cup \dots \cup \ti{t}^b_*$.
We denote by $\{\ti{t}_*\cup c_* \}$ the resulting Euler lift for $X$. Recall that this Euler lift represents $K(e)$.
We now have the following lemma:

\begin{lemma}\label{lem:tau1}
\[ \ctau\big(C_*^\varphi(X,Y;\F^d),\{c_* \otimes v_*\}\big)= (-1)^{\nu^\varphi} \ctau\big(C_*^\varphi(X;\F^d),\{( \ti{t}_*\cup c_*) \otimes v_*\}\big),\]
where
\[ \nu^\varphi= \sum_i \a_i^\varphi(X,Y)\a_{i-1}^\varphi(Y).\]
\end{lemma}

\begin{proof}
We consider the following  short exact sequence of chain complexes
\[ 0\to \bigoplus_{i=1}^b C_*^\varphi(Y_i;\F^d)\to C_*^\varphi(X;\F^d)\to C_*^\varphi(X,Y;\F^d)\to 0,\]
with the ordered bases
\[ \{t^{i}_* \otimes v_*\}_{i=1,\dots,b},\, \{( \ti{t}_*\cup c_*) \otimes v_*\} \mbox{ and } \{c_* \otimes v_*\}.\]
Note that these bases are in fact compatible.
Also note that by (\ref{equ:taue1}) and (\ref{equ:tau1}) we have
\[ \ctau(C_*^\varphi(Y_i;\F^d),\{t_*^i\otimes v_*\})=1.\]
It now follows from Lemma \ref{lem:ses} that
\[ \ctau\big(C_*^\varphi(X,Y;\F^d),\{c_* \otimes v_*\}\big)= (-1)^{\nu^\varphi} \ctau\big(C_*^\varphi(X;\F^d),\{(c_*\cup \ti{t}_*) \otimes v_*\}\big).\]
(Here we used that the complexes are acyclic, in particular the  $\mu$-term is zero and $\tau(\HH)=1$.)
\end{proof}

We now pick a basis $h_*$ for $H_*(X,Y;\R)$ which represents $\w$. We denote by $h_*^\d$ the dual basis of $H_*(X;\R)$,
and for $i=1,\dots,b$ we denote by $h_*^i$ a basis representing the canonical homology orientation of $Y_i$.
We denote by $\mathcal{H}$ the long exact sequence in homology of the pair $(X,Y)$. Note that the above bases in homology give rise to a basis $\BB_*$ for this complex.
We now have the following lemma:

\begin{lemma}\label{lem:tau2}
The following equality holds:
\[ \sign\big(\ctau(C_*(X,Y;\R),c_*,h_*)\big)= (-1)^{\nu+1}\cdot \sign\big(\ctau(C_*(X;\R), \ti{t}_* \cup c_*,h_*^\d \})\big)\cdot \sign(\tau(\HH_*,\BB_*)),\]
where
\[\ba{rcl} \nu&=& \sum_i \a_i(X,Y)\a_{i-1}(Y).\ea \]
\end{lemma}

\begin{proof}
We write
\[ C'_*=\bigoplus_{i=1}^b C_*(Y_i;\R), \,\, C_*=C_*(X;\R), \mbox{ and } C_*''=C_*(X,Y;\R).\]
We then consider the  short exact sequence of $0\to C'_*\to C_*\to C_*''\to 0$
with the bases
\[ \{t^{i}_* \}_{i=1,\dots,b},\, \{ \ti{t}_*\cup c_* \} \mbox{ and } \{c_* \}\]
and the homology bases
\[ \{h^i_*\}_{i=1,\dots,b},\,\, h_*^\d \mbox{ and } h_*.\]
Note that the bases of the  chain complexes are compatible.
Also note that by (\ref{equ:taue1}) and (\ref{equ:tau1}) we have
\[ \sign(\ctau(C_*(Y_i),\{t_*^i\},h_*^i))=1.\]
It now follows from Lemma \ref{lem:ses} that
\[ \ctau\big(C_*(X,Y;\R),\{c_* \}\big)= (-1)^{\nu+\mu}\cdot  \ctau\big(C_*(X;\R),\{(c_*\cup \ti{t}_*) \}\big)\cdot \tau(\HH_*,\BB_*),\]
where
\[ \mu=\sum_i \big((\b_i(C)+1)(\b_i(C')+\b_i(C''))+\b_{i-1}(C')\b_i(C'')\big). \]
Note that we only have to determine $\mu$ modulo 2.
Recall that $\chi(X,Y)=\chi(X)=\chi(Y)=0$ and note that $H_3(X;\R)=0$ since we assumed that $Y$ is non-empty.
One can now easily show that the following equalities hold modulo 2:
\[ \ba{rclrclrcl}
\b_0(C')&\equiv&b, & \b_0(C)&\equiv&1, &\b_0(C'')&\equiv&0, \\
\b_1(C')&\equiv&b, & \b_1(C)&\equiv&1+b_1(N), &\b_1(C'')&\equiv&b_1(N)-1, \\
\b_2(C')&\equiv&0, & \b_2(C)&\equiv&0, &\b_2(C'')&\equiv&1, \\
\b_3(C')&\equiv&0, & \b_3(C)&\equiv&0, &\b_3(C'')&\equiv&0.\ea \]
An elementary calculation now shows that $\mu\equiv 1$.
\end{proof}

Finally we also have the following lemma:

\begin{lemma} \label{lem:torh}
\[ \sign(\tau(\HH_*,\BB_*))=(-1)^{b_0(\partial N)+b_1(N)+1}.\]
\end{lemma}

\begin{proof}
We pick an orientation for $N$, i.e. we pick an orientation class $[N]\in H_3(N,\partial N;\Z)$  and we equip $\partial N$ with the orientation class
given by the image $[\partial N]$ of $[N]$ under the map $H_3(N,\partial N;\Z)\to H_2(\partial N;\Z)$.
 We will consider the intersection forms on $N$ and $\partial N$ with respect to these orientations.
We write
\[ L:=\im\{H_2(N,\partial N;\R)\to H_1(\partial N;\R)\}=\ker\{H_1(\partial N;\R)\to H_1(N;\R)\}.\]
It is well-known that $L$ is a half-dimensional subspace of $H_1(\partial N;\R)$ which is self-annihilating with respect to the intersection form on $\partial N$.
Put differently, we have  $L=L^\perp$.
We now pick a basis $v=\{v_1,\dots,v_b\}$ for $L$. We pick a subspace $L^*$ of $H_1(\partial N;\R)$ such that $H_1(\partial N;\R)=L\oplus L^*$.
Note that the intersection form on $H_1(\partial N;\R)$ gives rise to a non-singular pairing $L\times L^*\to \R$.
We now endow $L^*$ with the basis $v^*=\{v_1^*,\dots,v_b^*\}$ which is dual to $v$, i.e. which satisfies $v_i\cdot v_j^*=\delta_{ij}$.

We now consider the following two exact sequences (with $\R$ coefficients understood):
\[ \ba{cccccccccccccccccc}
\HH'_* \colon \hspace{-0.1cm}& 0\hspace{-0.1cm}&\to\hspace{-0.1cm}& L^*\hspace{-0.1cm}&\to \hspace{-0.1cm}&H_1(N)\hspace{-0.1cm}&\to\hspace{-0.1cm}&H_1(N,\partial N)\hspace{-0.1cm}&\to \hspace{-0.1cm}&H_0(\partial N)\hspace{-0.1cm}&\to \hspace{-0.1cm}&H_0(N)\hspace{-0.1cm}&\to \hspace{-0.1cm}&0, \\
\HH''_* \colon \hspace{-0.1cm}&0\hspace{-0.1cm}&\leftarrow \hspace{-0.1cm}&L\hspace{-0.1cm}&\leftarrow \hspace{-0.1cm}&H_2(N,\partial N)\hspace{-0.1cm}&\leftarrow \hspace{-0.1cm}&H_2(N)\hspace{-0.1cm}&\leftarrow \hspace{-0.1cm}&H_2(\partial N)\hspace{-0.1cm}&\leftarrow \hspace{-0.1cm}&H_3(N,\partial N)\hspace{-0.1cm}&\leftarrow \hspace{-0.1cm}&0.\ea \]
As before we equip the vector spaces $H_*(N;\R)$ and $H_*(N,\partial N;\R)$ with the bases $h_*^\d$ and $h_*$. We now  equip $H_0(\partial N;\R)$ with the basis consisting of a point in each
component, we equip $L$ and $L^*$ with the bases $v$ and $v^*$ and we equip $H_2(\partial N;\R)$ with the basis given by the orientation of $\partial N$.
We denote the resulting basis for $\HH'_*$ by $\BB'_*$ and the resulting basis for $\HH_*''$ by $\BB_*''$.

\begin{claim}
\[ \sign(\tau(\HH_*,\BB_*))=(-1)^b\cdot \sign(\tau(\HH'_*,\BB'_*))\cdot \sign(\tau(\HH_*'',\BB_*'')).\]
\end{claim}

Note that the obvious concatenation of $(\HH'_*,\BB'_*)$ and $(\HH_*'',\BB_*'')$ is exactly the same as $(\HH_*,\BB_*)$,
 except that the basings for the groups $H_0(\partial N;\R)$, $H_1(\partial N;\R)=L\oplus L^*$ and $H_2(\partial N;\R)$ differ.
 It thus suffices to show that $(-1)^b$ is the sign of the base change matrix of $H_*(\partial N;\R)$.
We denote by $k_*$ the basis for $H_*(\partial N;\R)$ described just before the claim.
For each $i=1,\dots,b$ we also denote by $(h_*^i)'$ a basis of $H_*(T_i;\R)$ which represents the homology orientation $\w'$  as defined in Lemma \ref{lem:canor}.
It follows easily from the definitions that
\[ \sign\big( [k_*/((h_*^1)'\cup \dots \cup (h_*^b)')]\big)=1.\]
On the other hand it follows from
Lemma \ref{lem:canor} that
\[ [h_*^i/(h_*^i)']=-1\]
for any $i$. The claim now follows  immediately from the definitions.

We denote by $(D_*,d_*)$ the based complex $\HH_*'$, in particular we have $D_0=H_0(N;\R)$ and $D_4=L^*$ and $\partial_i$ denotes the map $D_{i+1}\to D_i$.
We denote by $(D^*,d^*)$ the based complex given by $D^i:=D_{4-i}^*:={\hom_\R(D_{4-i},\R)}$, where the boundary map $\partial_i:D_{i+1}^*\to D_i^*$
is given by $\partial_{3-i}^*$, and where $d^*$ denotes the basis  dual to $d_*$.

\begin{claim}
The based chain complexes $(D^*,d^*)$ and $(\HH_*'',\BB_*'')$ are isomorphic, in particular
\[ \tau(D^*,d^*)=\tau(\HH_*'',\BB_*'').\]
\end{claim}

The proof of this claim follows from a careful consideration of Poincar\'e duality. Note that for any pair of spaces $(X,Y)$
we follow the convention of \cite[p.~332]{Br93} for the coboundary map $C^i(X,Y;\R)\to C^{i+1}(X,Y;\R)$  which says that
\be \label{conv:br} \delta f =(-1)^{i+1} f\circ \partial.\ee
In the diagram below, $D$ denotes the isomorphism given by Poincar\'e duality. Note that $D(a)\cap [N]=a$ for any $a\in H_i(N,\partial N)$ or
$a\in H_i(N)$.
We now consider the following diagram:
\[ \ba{cccccccccccccccccc}
& \hspace{-0.2cm}&H_1(\partial N)\hspace{-0.2cm}&\leftarrow \hspace{-0.2cm}&H_2(N,\partial N)\hspace{-0.2cm}&\leftarrow \hspace{-0.2cm}&H_2(N)\hspace{-0.2cm}&\leftarrow \hspace{-0.2cm}&H_2(\partial N)\hspace{-0.2cm}&\leftarrow \hspace{-0.2cm}&H_3(N,\partial N)\hspace{-0.2cm}&\leftarrow \hspace{-0.2cm}&0\\[1mm]
&&\downarrow D  &-&\downarrow D &&\downarrow D &-&\downarrow D&&\downarrow D\\[1mm]
 &\hspace{-0.2cm}& H^1(\partial N)\hspace{-0.2cm}&\leftarrow \hspace{-0.2cm}&H^1(N)\hspace{-0.2cm}&\leftarrow\hspace{-0.2cm}&H^1(N,\partial N)\hspace{-0.2cm}&\leftarrow \hspace{-0.2cm}&H^0(\partial N)\hspace{-0.2cm}&\leftarrow \hspace{-0.2cm}&H^0(N)\hspace{-0.2cm}&\leftarrow \hspace{-0.2cm}&0\\[1mm]
 &&\downarrow   &&\downarrow  &&\downarrow &-&\downarrow &&\downarrow \\[1mm]
  &\hspace{-0.2cm}& H_1(\partial N)^*\hspace{-0.2cm}&\leftarrow \hspace{-0.2cm}&H_1(N)^*\hspace{-0.2cm}&\leftarrow\hspace{-0.2cm}&H_1(N,\partial N)^*\hspace{-0.2cm}&\leftarrow \hspace{-0.2cm}&H_0(\partial N)^*\hspace{-0.2cm}&\leftarrow \hspace{-0.2cm}&H_0(N)^*\hspace{-0.2cm}&\leftarrow \hspace{-0.2cm}&0.\ea \]
Here the top vertical maps are given by Poincar\'e duality and the bottom vertical maps are given by evaluation.
A minus sign in a square indicates that the corresponding square anti-commutes, and all other squares commute.
That the above mentioned (anti-) commutativity holds, follows from  \cite[Theorem~VI.9.2]{Br93} for the upper sequence of squares,
and it follows from the definitions and from convention (\ref{conv:br}) for the lower sequence of squares.

Note that the second and fifth vertical maps define the following isomorphism:
\[ \ba{rcl} H_i(N,\partial N;\R) &\to & H_{3-i}(N;\R)^* \\
a&\mapsto & (b\mapsto D(a)\cap b).\ea \]
By \cite[Theorem~VI.5.2]{Br93} we have
\[ \ba{rcl} D(a)\cap b&=&D(a)\cap (D(b)\cap [N])=(D(a)\cup D(b))\cap [N]\\
&=&(D(a)\cup D(b))\cap [N]=(D(b)\cup D(a))\cap [N]\\
&=&b\cdot a, \ea \]
where $b\cdot a$ indicates the intersection pairing as defined in Section \ref{section:dualhomol}. The same arguments apply for all other vertical maps.
Note for the four intersection pairings defined on the right we have  $b\cdot a=(-1)^{|a|\cdot|b|}a\cdot b=a\cdot b$ since one of the dimensions is always even.
On the other hand the intersection pairing $H_1(\partial N;\R)\times H_1(\partial N;\R)\to \R$ defined by the vertical maps on the left is anticommutative.
The above discussion shows that the following diagram, where the  vertical maps are all given by sending $a$ to the map $b\mapsto (a\cdot b)$, commutes:
\[ \ba{cccccccccccccccccc}
&H_1(\partial N)\hspace{-0.2cm}&\leftarrow \hspace{-0.2cm}&H_2(N,\partial N)\hspace{-0.2cm}&\leftarrow \hspace{-0.2cm}&H_2(N)\hspace{-0.2cm}&\leftarrow \hspace{-0.2cm}&H_2(\partial N)\hspace{-0.2cm}&\leftarrow \hspace{-0.2cm}&H_3(N,\partial N)\hspace{-0.2cm}&\leftarrow \hspace{-0.2cm}&0\\[1mm]
&\downarrow   &&\downarrow  &&\downarrow  &&\downarrow &&\downarrow \\[1mm]
 & H_1(\partial N)^*\hspace{-0.2cm}&\leftarrow \hspace{-0.2cm}&H_1(N)^*\hspace{-0.2cm}&\leftarrow\hspace{-0.2cm}&H_1(N,\partial N)^*\hspace{-0.2cm}&\leftarrow \hspace{-0.2cm}&H_0(\partial N)^*\hspace{-0.2cm}&\leftarrow \hspace{-0.2cm}&H_0(N)^*\hspace{-0.2cm}&\leftarrow \hspace{-0.2cm}&0.\ea \]
This diagram now gives rise to the following commutative diagram, where again the  vertical maps are all given by sending $x$ to the map $y\mapsto (x\cdot y)$:
\[ \ba{cccccccccccccccccc}
\HH_*'':&0\hspace{-0.2cm}&\leftarrow \hspace{-0.2cm}&L\hspace{-0.2cm}&\leftarrow \hspace{-0.2cm}&H_2(N,\partial N)\hspace{-0.2cm}&\leftarrow \hspace{-0.2cm}&H_2(N)\hspace{-0.2cm}&\leftarrow \hspace{-0.2cm}&H_2(\partial N)\hspace{-0.2cm}&\leftarrow \hspace{-0.2cm}&H_3(N,\partial N)\hspace{-0.2cm}&\leftarrow \hspace{-0.2cm}&0\\[1mm]
&&&\downarrow   &&\downarrow  &&\downarrow  &&\downarrow &&\downarrow \\[1mm]
(\HH_*')^*  &0\hspace{-0.2cm}&\leftarrow\hspace{-0.2cm}& (L^*)^*\hspace{-0.2cm}&\leftarrow \hspace{-0.2cm}&H_1(N)^*\hspace{-0.2cm}&\leftarrow\hspace{-0.2cm}&H_1(N,\partial N)^*\hspace{-0.2cm}&\leftarrow \hspace{-0.2cm}&H_0(\partial N)^*\hspace{-0.2cm}&\leftarrow \hspace{-0.2cm}&H_0(N)^*\hspace{-0.2cm}&\leftarrow \hspace{-0.2cm}&0.\ea \]
It is straightforward to verify, that the vertical maps define isomorphisms of based vector spaces. (For example, the left vertical map sends $v_i$ to $w\mapsto v_i\cdot w$, but by definition we have $v_i\cdot v_j^*=\delta_{ij}$.
For the vertical maps on the right note again that the intersection pairings are commutative.)
i.e. the maps define an isomorphism  $(D^*,d^*)$ and $(\HH_*'',\BB_*'')$ as claimed.
This concludes the proof of the claim.

Now also recall that previously we defined a based complex $(D^\d,d^\d)$. Note that for each $i$ we have the following commutative diagram
\[ \xymatrix{ D_i^* \ar[d]^= \ar[rr]^{\partial_{3-i}^*} && D_{i-1}\ar[d]^= \\
D_i^\d \ar[rr]^{(-1)^{4-i}\partial_{3-i}^*} && D_{i-1}^\d.}\]
It is now straightforward to see that
\[ \tau(D^*,d^*)=(-1)^{b_1(N)-b+1}\tau(D^\d,d^\d).\]

On the other hand it follows from Lemma \ref{lem:dual} that
\[ \tau(D^\d,d^\d)=(-1)^r\tau(D_*,d_*)^{-1}\]
where
\[
 r=\sum_{i=0}^4 \big(\a_i(D)\a_{i-1}(D)+\b_i(D)\b_{i-1}(D)\big)+\sum_{i=0}^{2}\big(\a_{2i}(D)+\b_{2i}(D)\big). \]
 Note that $D$ is acyclic, and therefore an elementary calculation now shows that
\[  r=\sum_{i=0}^4\a_i(D)\a_{i-1}(D)+\sum_{i=0}^{2}\a_{2i}(D)\equiv b\mbox{ mod }2.\]

Combining the above considerations we obtain that
\[ \ba{rcl}
\sign(\tau(\HH_*,\BB_*))&=&(-1)^b\cdot \sign(\tau(\HH'_*,\BB'_*))\cdot \sign(\tau(\HH_*'',\BB_*''))\\
&=&(-1)^b\cdot \tau(D_*,d_*)\cdot \tau(D^*,d^*)\\
&=&(-1)^{b}\cdot (-1)^{b_1(N)-b+1}\cdot \tau(D_*,d_*)\cdot \tau(D^\d,d^\d)\\
&=&(-1)^{b}\cdot (-1)^{b_1(N)-b+1}\cdot (-1)^{b}\cdot \tau(D_*,d_*)\cdot \tau(D_*,d_*)\\
&=&(-1)^{b+b_1(N)+1}.\ea \]
\end{proof}

\begin{proof}[Proof of Theorem \ref{thm:torsionboundary}]
Recall that
\[ \a_i^\varphi(Y)=d\cdot \a_i(Y) \mbox{ and }\a_i^\varphi(X,Y)=d\cdot \a_i(X,Y),\]
in particular $\nu^\varphi=d\cdot \nu$. It now follows from this observation, the  definitions, Lemmas \ref{lem:tau1}, \ref{lem:tau2} and \ref{lem:torh}
that
\[  \ba{cl}
&\tau(N,\partial N,\varphi,e,\w)\\
=&\ctau\big(C_*^\varphi(X,Y;\F^d),\{c_* \otimes v_*\}\big)\cdot \sign\big(\ctau(C_*(X,Y;\R),c_*,h_*)\big)^d\\
=& (-1)^{-d} \cdot \ctau\big(C_*^\varphi(X;\F^d),( \ti{t}_*\cup c_*) \otimes v_*\big)\cdot \sign\big(\ctau(C_*(X;\R), \ti{t}_* \cup c_*,h_*^\d )\big)^d\cdot \sign(\tau(\HH_*,\BB_*))^d\\
=&(-1)^{d(b_0(\partial N)+b_1(N))} \cdot \ctau\big(C_*^\varphi(X;\F^d),( \ti{t}_*\cup c_*) \otimes v_*\big)\cdot \sign\big(\ctau(C_*(X;\R), \ti{t}_* \cup c_*,h_*^\d )\big)^d\\
=& (-1)^{d(b_0(\partial N)+b_1(N))}\cdot  \tau(N,\varphi,K(e),\w^\d).\ea \]
\end{proof}

\section{Proof of the main duality theorem}
\label{sec:proof}

\subsection{Spin-c structures}

We first recall some facts about the set of $\spinc$-structures on 3-manifolds with empty or toroidal boundary.
We refer to \cite[Section~XI.1]{Tu02} for full details.
Let $N$ be a  3-manifold with empty or toroidal boundary.
We denote by $\spinc(N)$ the set of $\spinc$-structures on $N$.
The set $\spinc(N)$ admits a canonical, free and transitive action by $\HH_1(N)$.
Also note that given $\ss\in \spinc(N)$ we can consider the Chern  class $c_1(\ss)\in H^2(N,\partial N;\Z)=\HH_1(N)$. The Chern class has the property,  that for any $\ss\in \spinc(N)$ and $h\in \HH_1(N)$ we have
\be \label{equ:c1square} c_1(h\cdot \ss)=h^2c_1(\ss).\ee
Turaev showed that
there exists a canonical
$\HH_1(N)$-equivariant bijection
\[ \spinc(N)\to \vect(N)\]
which  preserves  Chern classes.
(An explicit description of the inverse map is provided in \cite{BP01}).
We will use this canonical bijection to identify $\vect(N)$ with $\spinc(N)$.
In particular, given a representation  $\varphi\colon \pi_1(N)\to \gl(d,\FF)$ over a field $\FF$, $\w\in \ort(N)$
and  $\ss \in \spinc(N)$ we now pick $e\in \vect(N)$ which corresponds to $\ss$ under the above canonical bijection and we define
\[ \tau(N,\varphi,\ss,\w):=\tau(N,\varphi,e,\w).\]

\subsection{Proof of Theorem \ref{mainthmstrongintro}}
In this subsection, we prove Theorem~\ref{mainthmstrongintro}.

\begin{proof}[Proof of Theorem~\ref{mainthmstrongintro}]
We pick $e\in \eul(N)=\vect(N)$ which corresponds to $\ss$.
If $N$ is a closed 3-manifold, then the conclusion follows immediately from  Theorem \ref{thm:dualitygeneralclosed} and the fact that $z(N)=0$.
We now consider the case that $N$ has non-empty toroidal boundary. Then it follows from  Theorem \ref{thm:dualitygeneral} that
\[  \ol{\tau(N,\varphi,e,\w)}=(-1)^{ds(N)}\cdot {\tau(N,\partial N,\varphi^\d,e^\d,\w^\d)}.\]
where
\[  s(N)=\sum_{i=0}^m \b_i(N)\b_{i-1}(N)+\sum_{i=0}^{[m/2]}\b_{2i}(N). \]
A straightforward calculation, using that $\chi(N)=0$ and that $b_3(N)=0$, shows that $s(N)=b_1(N)$.
But it follows from  Theorem \ref{thm:torsionboundary} that
\[ {\tau(N,\partial N,\varphi^\d,e^\d,\w^\d)}=(-1)^{d(b_1(N)+b_0(\partial N))} \cdot \tau(N,\varphi^\d,K(e^\d),\w).\]
Recall that $c_1(e)\in \HH_1(N)$ is defined to be the unique element such that $K(e^\d)=c_1(e)^{-1}\cdot e$.
Combining the above we now see that
\[  \ol{\tau(N,\varphi,e,\w)}=(-1)^{db_0(\partial N)} \cdot \tau(N,\varphi^\d,c_1(e)^{-1}\cdot e,\w).\]
The theorem now follows from Lemma \ref{lem:indet}.
\end{proof}

\section{Applications to twisted Alexander polynomials}\label{section:twialex}

\subsection{Definition of twisted Alexander polynomials}

Let $N$ be a 3-manifold with empty or toroidal boundary. We write $\pi=\pi_1(N)$. Let $\a\colon \pi\to \gl(d,\KK)$ be a representation over a field $\KK$. Let $\phi\colon \pi\to F$ be an epimorphism onto a free abelian group.
We denote the quotient field of $\KK[F]$ by $\KK(F)$.
Note that we have a tensor representation
\[ \ba{rcl} \a\otimes \phi\colon \pi&\to &\aut(\KK(F)^d)\\
g&\mapsto & (v\mapsto \phi(g)\cdot \a(g)(v))\ea \]
(here we view $\a(g)\in \gl(d,\KK)$ as an automorphism of $\KK(F)^d$ in the obvious way).
Note that for $g\in \pi$ we have
\[ \det((\a\otimes \phi)(g))=\phi(g)^d\det(\a(g)).\]

Given $(N,\a,\phi)$ as above we denote by
$\tau(N,\a\otimes \phi)\subset \KK(F)$ the set of values
\[ \{\tau(N,\a\otimes \phi,\ss,\w)\,|\,\w\in\ort(N) \mbox{ and }\ss\in \spinc(N)\}.\]
This set is called the \emph{twisted Alexander polynomial of $(N,\a,\phi)$}.
Each $\tau(N,\a\otimes \phi,\ss,\w)$ is called a \emph{representative} of $\tau(N,\a\otimes \phi)$.
Note that by Lemma \ref{lem:indet} any two representative differ by multiplication by an element in the set
\[ \ba{lll} \{ \eps\cdot f^d\cdot d\,|\, \eps=\pm 1, f\in F, d\in \det(\a(\pi))\},
&\mbox{ if $n$ is odd},\\
 \{  f^d\cdot d\,|\,f\in F, d\in \det(\a(\pi))\}&\mbox{ if $n$ is even}.\ea \]
In particular, if $\varphi\colon \pi\to \sl(2d,\KK)$ is a special linear, even dimensional representation,
then $\tau(N,\a\otimes \phi)\in \KK(F)$ is well-defined up to multiplication by an $2n$-th power of an element in $F$.

\subsection{Properties of Chern classes}

We now state a few properties of the  Chern class which we will need in the next section.

\begin{lemma}\label{lem:c1}
Let $N$ be a 3-manifold with empty or toroidal boundary and let $v\in \vect(N)$.
Then the following hold:
\bn
\item For any $h\in \HH_1(N)$ we have $c_1(hv)=h^2\cdot c_1(v)$,
\item if $N$ is closed, then there exists $w\in \vect(N)$ such that  $c_1(w)$ is trivial,
\item if $N$ is closed, then $c_1(v)=g^2$ for some $g\in \HH_1(N)$.
\en
\end{lemma}

\begin{proof}
The first statement follows easily from the definitions.
Now assume that $N$ is  a closed 3-manifold.
Recall  that any closed 3-manifold $N$ is parallelizable, in particular $N$ admits a vector field $w$ which is isotopic to $-w$. We denote the corresponding  $\spinc$-structure by $w$ as well. Note that $c_1(w)$ is trivial in $\HH_1(N)$ by definition.
We can now write $v=gw$ for some $g\in \HH_1(N)$ and it follows from (1) and (2) that $c_1(v)=c_1(g\cdot w)=g^2$.
\end{proof}

For links in homology spheres the following result of Turaev
(see \cite[Section~VI.2.2]{Tu02}) classifies which classes are realized by
Chern classes of elements in $\vect(N)$.

\begin{proposition}\label{prop:c1link}
Let $L$ be an oriented, ordered $m$-component link in a $\Z$-homology sphere $N$. We write $N_L=N\sm \nu L$.
We identify $\HH_1(N_L)$ with the free abelian multiplicative group generated by $t_1,\dots,t_m$.
Then for any $v\in \vect(N_L)$ we have that $c_1(v)$ is a charge, and conversely for any charge $c$ there exists a $v\in \vect(N_L)$ such that $c_1(v)=c$.
\end{proposition}

\subsection{Proofs of Theorems \ref{mainthm:twialexintro}, \ref{mainthm:twialexlinkintro} and \ref{thm:evendegreeintro}}

Recall that if $\KK$ is a field with involution, then we endow $\KK(F)$ with the involution induced by the involution on $\KK$ and the involution given by $\ol{f}=f^{-1}$ for $f\in F$.
The following is now an almost immediate consequence of
Theorem \ref{mainthmstrongintro}.

\begin{theorem}\label{mainthm:twialex}
Let $N$ be a  3-manifold with empty or toroidal boundary, let  $\a\colon \pi_1(N)\to \gl(d,\KK)$ be a representation
over a field $\KK$ with  involution, and let $\phi\colon H_1(N;\Z)\to F$
be an admissible epimorphism onto a free abelian group.
Suppose that $\a$ is conjugate to its dual.
Then for any $\ss\in \spinc(N)$ and any $\w\in \ort(N)$ we have
\[ \ol{\tau(N,\a\otimes \phi,\ss,\w)}=(-1)^{db_0(\partial N)}\cdot \det(\a(c_1(\ss)))\cdot \phi(c_1(\ss))^d\cdot \tau(N,\a\otimes \phi,\ss,\w). \]
\end{theorem}

\begin{proof}
An elementary argument shows that $\a\otimes \phi$ as a representation over $\FF=\KK(F)$ is conjugate to its dual.
Note that an elementary calculation using the fact that $\phi$ restricted to any boundary component is non-trivial shows that
\[ H_*^{\a\otimes \phi}(\partial N;\KK(F)^d)=0.\]
It now follows from
Theorem \ref{mainthmstrongintro} that
\[ \ba{rcl} \ol{\tau(N,\a \otimes \phi,\ss,\w)}&=&(-1)^{db_0(\partial N)}\cdot \det((\a\otimes \phi)(c_1(\ss)))\cdot  \tau(N,{(\a\otimes \phi)^\d},\ss,\w)\\
&=&(-1)^{db_0(\partial N)}\cdot \det(\a(c_1(\ss)))\cdot \phi(c_1(\ss))^d \cdot  \tau(N,\a\otimes \phi,\ss,\w).\ea \]
\end{proof}

The first statement of Theorem \ref{mainthm:twialexintro} now follows immediately, and the second statement of Theorem \ref{mainthm:twialexintro}
if we apply the above theorem to the $\spinc$-structure with trivial Chern class which exists on closed 3-manifolds by Lemma \ref{lem:c1}.
 Theorem \ref{mainthm:twialexlinkintro} is an immediate consequence of the above theorem and of Proposition \ref{prop:c1link}.
Now we prove theorem~\ref{thm:evendegreeintro}:

\begin{proof}[Proof of Theorem~\ref{thm:evendegreeintro}]
Let $S$ be a Thurston norm minimizing surface which is dual to $\phi$. Since $N$ is irreducible and since $N\ne S^1\times D^2$ we can arrange that $S$ has no disk components.
 We therefore have
\[ d\cdot x(\phi)\equiv d\cdot\chi_-(S) \equiv d\cdot b_0(\partial S)\, \mod\, 2.\]  Then for any $\ss\in\spinc (N)$, by \cite[Lemma~VI.1.2]{Tu02}, we have
$$
d\cdot b_0(\partial S) \equiv d\cdot(c_1(\ss)\cdot S) \,\mod\, 2
$$
where $c_1(\ss)\cdot S$ is the intersection number of $c_1(\ss)\in H_1(N)$ with $S$. Since $S$ is dual to $\phi$, we obtain that $d\cdot (c_1(\ss)\cdot S)  = d\cdot \phi(c_1(\ss))$.
 Therefore it suffices to show that $d\cdot \phi(c_1(\ss))\equiv \deg(\tau(N,\a\otimes \phi, \ss, \w))\,\mod\, 2$ for any homology orientation $\w\in \ort (N)$.

Suppose that $\tau(N,\a\otimes \phi, \ss, \w) = f(t)/g(t)$ for some $f(t), g(t)\in \KK[t^{\pm 1}]$, where $f(t) = \sum\limits_{i=n}^m a_it^i$ with $a_n,a_m\ne 0$ and $g(t) = \sum\limits_{j=q}^p b_jt^j$ with $b_p,b_q\ne 0$. In particular, $\deg(\tau(N,\a\otimes \phi, \ss, \w)) = (m-n)-(p-q)$. Then by Theorem~\ref{mainthm:twialex}, one easily obtains that
$$
\ol{f(t)}g(t) = c\cdot \phi(c_1(\ss))^d f(t) \ol{g(t)}
$$
for some $c\in \KK$. By looking at the highest exponents of both sides,  we now have $-n+p=d\cdot \phi(c_1(\ss))+ m -q$. Therefore
$$
d\cdot \phi(c_1(\ss)) \equiv -m-n+p+q \equiv \deg(\tau(N,\a\otimes \phi, \ss, \w))\,\mod\, 2.
$$

\end{proof}

\subsection{$\sltwoc$-representations and twisted Alexander polynomials}\label{section:sl2}
In this subsection we prove Theorem~\ref{thm:sltwocintro}.
We recall the following well-known lemma:
\begin{lemma}
Let $N$ be a 3-manifold with empty or toroidal boundary. Then a representation $\a:\pi_1(N)\to \sltwoc$ is conjugate to its dual.
\end{lemma}

\begin{proof}
 We equip $\C^2$ with the bilinear form
$(v,w)\mapsto \det(v\, w)$. It is clear that $\sl(2,\C)$ acts by isometries on the form,
hence by \cite[Lemma~3.1]{HSW10} any representation $\a\colon \pi_1(N)\to \sl(2,\C)$ is dual to its conjugate.
\end{proof}

\begin{proof}[Proof of Theorem~\ref{thm:sltwocintro}]
Since $\a$ is irreducible it follows from Theorem \ref{thm:taupoly} that any loose representative of  $\tau(N,\a\otimes \phi)$ lies in $\ct$.
By the previous lemma the representation $\a$ is conjugate to its dual.
Let $\ss\in \spinc(N)$ and $\w\in\ort(N)$. It follows from
Theorem \ref{mainthm:twialex} that
\[ \ol{\tau(N,\a\otimes \phi,\ss,\w)}=  \phi(c_1(\ss))^2\cdot \tau(N,\a\otimes \phi,\ss,\w). \]
We thus see that
\[ \tau:= \phi(c_1(\ss))\cdot \tau(N,\a\otimes \phi,\ss,\w)\]
has the desired properties.
\end{proof}

We conclude this section with a short excursion into the representation theory of $\sltwoc$.
Given a vector space $V$ of dimension $n$ we define
\[ \sym^k(V):=V^{\otimes k}/\sim \]
where $\sim$ is the equivalence relation generated by
\[ v_1\otimes \dots \otimes v_i\otimes \dots \otimes v_j \otimes \dots \otimes v_k=
v_1\otimes \dots \otimes v_j\otimes \dots \otimes v_i \otimes \dots \otimes v_k.\]
Note that $\dim(\sym^k(V))$ equals the number of elements in the set
\[ P_k:=\{ (e_1,\dots,e_k)\, |\, 1\leq e_1\leq e_2\leq \dots \leq e_k\leq n\}.\]
Put differently, $\dim(\sym^k(V))$ equals the number of Young tableaux of size $n$.
Note that if $n=2$, then $\dim(\sym^k(V))=k+1$.

Note that given any $k$ there exists a canonical representation
\[ \ba{rcl} \rho_k\colon \aut(V)&\to & \aut(\sym^k(V))\\
\varphi &\mapsto & (v_1\otimes \dots \otimes v_k\mapsto \varphi(v_1)\otimes \dots \otimes \varphi(v_k)).\ea \]
It is well-known that any irreducible representation of $\sltwoc$ is isomorphic to one of the $\rho_k$.
Given a representation $\a\colon \pi\to \sltwoc$ we denote the representation
$\pi\to \aut(\C^2)\xrightarrow{\rho_k} \aut(\sym^k(\C^2))$ by $\a_k$.
Note that $\a_1=\a$ and that $\a_2$ equals the adjoint representation.

Given a 3-manifold $N$ and a representation $\a:\pi_1(N)\to \sltwoc$ the torsions and twisted Alexander polynomials corresponding to
the representations $\a_k, k\geq 2$ have recently been studied in detail, see e.g. \cite{DY09} and
see \cite{Mu09,Mu10} for the relationship of such torsions to the volume of a hyperbolic 3-manifold.
The following well-known lemma says that Theorem \ref{mainthm:twialex} also applies to such representations.

\begin{lemma}\label{lem:akconjdual}
Let $\KK$ be a field with involution with  characteristic zero, and let $\varphi:\pi_1(N)\to \gl(d,\KK)$ be a representation which is conjugate to its dual, then
 $\a_k$ is also conjugate to its dual.
 \end{lemma}

\begin{proof}
We write $V=\KK^d$. By   \cite[Lemma~3.1]{HSW10} there exists a non-degenerate sesquilinear
 form $\ll\,,\,\rr$ on $\KK^d$ such that $\a$ acts by isometries.
 Now consider the form
 \[ \ba{rcl} \sym^k(V)\times \sym^k(V) &\to & \KK \\
( v_1\otimes \dots \otimes v_k , w_1\otimes \dots \otimes w_k)&\mapsto &
\sum_{\s\in S_k}\prod_{i=1}^k\ll v_i,w_{\s(i)}\rr.\ea \]
It is clear that this is a sesquilinear form on $\sym^d(V) $ and that $\a_k$ preserves the form.
It remains to show that this form is non-singular.

Let $v_1,\dots,v_d$ be a basis for $V$. Since $\ll\,,\,\rr$  is non-degenerate we can find vectors $w_1,\dots,w_d$ such that $\ll v_i,w_j\rr=\delta_{ij}$. Given $e=(e_1,\dots,e_k)\in P_k$ we define
\[ v_e:=v_{e_1}\otimes \dots \otimes v_{e_k}\mbox{ and }w_e:=w_{e_1}\otimes \dots \otimes w_{e_k}.\]
It is straightforward to verify that given $e,f\in P_k$ we have $\ll v_e,w_f\rr=0$ if $e\ne f$
and $\ll v_e,w_f\rr\ne 0$ otherwise. (Note that for the latter conclusion we used the assumption that $\KK$ has characteristic zero.)
\end{proof}

\section{Reidemeister torsion for irreducible non-abelian representations}
\label{section:polynomial}

Let $N$ be a 3-manifold with empty or toroidal boundary, $\a\colon\pi=\pi_1(N)\to \gl(d,\KK)$ a representation  over a field $\KK$ and $\phi\colon\pi\to F$ an epimorphism onto a free abelian group.
The Reidemeister torsion $\tau(N,\a\otimes \phi)\in \KK(F)$ is in general a rational function.
On the other hand the following theorem shows  that in many cases $\tau(N,\a\otimes \phi)$ actually lies in $\KK[F]$.

\begin{theorem}
\label{thm:taupoly}
Let $N$, $\a$, and $\phi$ be as above.
Suppose one of the following holds:
\bn
\item $\mbox{rank}(F)\geq 2$, or
\item $\a$ is  an irreducible  representation
which is non--trivial when restricted to $\ker(\phi)$,
\en
then $\tau(K,\a\otimes \phi)\in \KK(F)$ is a polynomial, i.e. it lies in $\KK[F]$.
\end{theorem}

If $\mbox{rank}(F)\geq 2$  then the conclusion follows from
earlier work by two of the authors \cite[Lemma~6.2,~Lemma~6.5,~Theorem~6.7]{FK08b}), see also
 \cite[Proposition~9]{Wa94}.
The second part generalizes a result by Wada (see  \cite[Proposition~8]{Wa94}).
 We also refer to\cite[Theorem~1.1]{KM05} for an interesting result regarding twisted torsion of knots for $\sl(2,\C)$-representations.

The proof of the second part of the theorem will require the remainder of this section.
Note that by (1) it suffices to consider the case that $\phi$ is an epimorphism onto $\Z$.
We find it convenient to rephrase the theorem in the language of twisted Alexander polynomials, the definition of which we now recall.
Let $X$ be  a topological space, $\phi\colon\pi_1(X)\to \Z$ an epimorphism  and
$\a\colon\pi_1(X)\to \gl(d,\KK)$ a representation over a  field $\KK$. Recall that we can now define a tensor
representation $\a\otimes \phi\colon\pi_1(X)\to \gl(d,\KK\tpm)$.
We obtain a twisted module  $H_i^{\a\otimes \phi}(X;\KK\tpm^d)$ over the ring $\KK\tpm$.
 Note that  $H_i^{\a}(X;\KK\tpm^d)$ is a finitely generated module over
 $\KK\tpm$. We now denote by $\Delta^{\a}_{X,\phi,i}\in \KK\tpm$ the order of
  $H_i^{\a\otimes \phi}(X;\KK\tpm^d)$ and refer to it as the \emph{$i$--th twisted Alexander order of $(X,\phi,\a)$}.
 We refer to \cite{Tu01} or \cite[Section~2]{FV10}  for the precise definitions.
Note that the twisted Alexander polynomials are well--defined up to multiplication by a unit in the ring $\KK\tpm$.

Twisted  torsions and twisted Alexander polynomials are closely related invariants
as the following well--known proposition  shows.
(We refer to \cite{KL99}, \cite[Theorem~4.7]{Tu01} and \cite[Proposition~2.5~(1)~and~Lemma~2.8]{FK06} for details.)

\begin{proposition}\label{prop:deltataualpha}
Let $N$ be a 3-manifold with empty or toroidal boundary. Let
$\a\colon\pi=\pi_1(N)\to \gl(d,\KK)$ be  a  representation over a field  $\KK$.
We denote the dual of $\a$ by $\b$.
Suppose that $ \tau(N,{\a}\otimes \phi )\ne 0$.
Then the following hold:
\bn
\item If $N$ is closed, then  $\Delta_{N,\phi,0}^\a\neq 0$,  $\Delta_{N,\phi,0}^\b\neq 0$
and
\[ \tau(N,{\a}\otimes \phi )=  \frac{\Delta^{{\a}}_{N,\phi,1}}{\Delta^{{\a}}_{N,\phi,0}\cdot \Delta_{N,\phi,0}^\b}\in \KK(t)\]
up to multiplication by a unit in $\KK\tpm$.
\item If $N$ has non-trivial boundary, then  $\Delta_{N,\phi,0}^\a\neq 0$
and
\[ \tau(N,{\a}\otimes \phi )=  \frac{\Delta^{{\a}}_{N,\phi,1}}{\Delta^{{\a}}_{N,\phi,0}}\in \KK(t)\]
up to multiplication by a unit in $\KK\tpm$.
\en
\end{proposition}

Our main technical result of this section is now the following proposition, which we phrase in a slightly more general language than strictly necessary, hoping that the lemma will be of independent interest.

\begin{proposition}\label{prop:delta0}
Let $X$ be a topological space. We write $\pi=\pi_1(X)$.
 Suppose $\phi\colon\pi\to \Z$ is a non-trivial homomorphism such that $\ker(\phi)$ is non--trivial and let $\a\colon\pi\to \gl(d,\KK)$ be an  irreducible representation over a field $\KK$
 which is non--trivial when restricted to $\ker(\phi)$.
Then $ \Delta_{X,\phi,0}^\a$ is a unit in $\KK\tpm$.
\end{proposition}

Note that  Theorem \ref{thm:taupoly} is now an immediate consequence of Propositions
\ref{prop:deltataualpha}  and  \ref{prop:delta0}.

\begin{proof}[Proof~of~Proposition~\ref{prop:delta0}]
Suppose that  $\a\colon\pi\to \gl(d,\KK)$ is a representation over a field $\KK$
which is non--trivial when restricted to $\ker(\phi)$ and
such that $ \Delta_{X,\phi,0}^\a\neq 1\in \KK\tpm$ (up to multiplication by a unit in $\KK\tpm$). We will show that $\a$ is not irreducible.
 We write $\G=\ker(\phi)$ and we pick $\mu\in \pi$ such that $\phi(\mu)$ generates the image of $\phi$.
We denote by $\a\otimes \phi\colon\pi\to \aut(\KK\tpm^d)$ the tensor representation.
First recall that $\Delta_{X,\phi,0}^\a=1$ if and only if $H_0^{\a\otimes \phi}(\pi;\KK\tpm^d)=0$
(cf. e.g. \cite[Lemma~2.2]{FK06}).
Also recall (cf. \cite[Section~VI]{HS97}) that
\[ H_0^{\a\otimes \phi}(\pi;\KK\tpm^d) =\KK\tpm^d/((\a \otimes \phi)(g)(v)-v \, |\, v\in \KK\tpm^d, g\in \pi).\]
By our assumption we have $H_0^{\a\otimes \phi}(\pi;\KK\tpm^d)\neq 0$. It is straightforward to see that this implies that
$H_0^\a(\G;\KK^d)=\KK^d/(\a(g)(v)-v \, |\, v\in \KK^d, g\in \G)$ is also non--trivial.
Now pick a non--singular form $\ll \, ,\, \rr$ on $\KK^d$ and denote by $\a^\d\colon\pi\to \gl(d,\KK)$ the unique representation which satisfies
$\ll \a^\d(g)v,\a(g)w \rr =\ll v,w\rr$ for all $g\in \pi, v,w\in \KK^d$.
We now let $Y=K(\G,1)$ and we denote by $\widetilde{Y}$ the universal cover of $Y$.
We write $(\KK^d)_\a$ to denote $\KK^d$ with the $\G$--action given by $\a$, and similarly we write $(\KK^d)_{\a^\d}$.
Using the inner
product we then get an isomorphism of $\KK$--module chain complexes:
\[
\ba{rcl} \hom_{\Z[\G]}(C_*(\widetilde{Y}),(\KK^d)_{\a^\d})&\to&
\hom_\KK\big(C_*(\wti{Y};
(\KK^d)_\a),\KK\big)
=\hom_\KK\big(C_*(\wti{Y}\otimes_{\Z[\G]} (\KK^d)_\a,\KK\big)\\
f&\mapsto& \left((c\otimes w)\mapsto \langle f(c),w\rangle\right). \ea
\]
Note that this map is well--defined since $\langle \b(g^{-1})v,w\rangle =\langle
v,\ol{\b}(g)w\rangle$. It is now easy to see that this defines in fact an isomorphism of
$\KK$--module chain complexes.
It now follows that
\[ \ba{rcl} H^i_{\a^\d}(\G;\KK^d)&=& H_i(\hom_{\Z[\G]}(C_*(\wti{Y}),(\KK^d)_{\a^\d}))\\
&\cong & H_i(\hom_\KK\big(C_*(\wti{Y};
(\KK^d)_\a),\KK\big))\\
&\cong & H_i(C_*(\wti{Y};
(\KK^d)_\a\big))\\
&=&H_i^\a(\G;\KK^d).\ea\]
Note that the second to last isomorphism is given by  the universal coefficient theorem.
Recall (cf. again \cite[Section~VI]{HS97}) that
\[ H^0_{\a^\d}(\G;\KK^d) =\{ v\in \KK^d\, |\, \a^\d(g)(v)=v\, \mbox{ for all }g\in \G, v\in \KK^d\}.\]
We now write $V:=\{ v\in \KK^d\, |\, \a^\d(g)(v)=v\, \mbox{ for all }g\in \G, v\in \KK^d\}$ and we let $W\subset \KK^d$ be the orthogonal complement of $V$. In particular  $V\oplus W=\KK^d$. Note that $V$ is non--trivial by assumption and note that $W$ is non--trivial since $\G=\ker(\phi)$ is non--trivial and since $\a$ (and hence $\a^\d$) is non--trivial when restricted to $\G$ by our assumption.

Note that with respect to the decomposition $\KK^d=V\oplus W$ we have
\[ \a^\d(g)=\bp \id & * \\ 0& *\ep \]
for any $g\in \G$. In particular  there exist maps $A\colon\pi\to \hom(W,V)$ and $B\colon\pi\to \endo(W)$ such that
\[ \a^\d(g)=\bp \id & A(g) \\ 0 &B(g) \ep \]
for any $g\in \G$.
We now write
\[ \a^\d(\mu)=\bp C&D\\ E&F \ep\]
with $C\in \endo(V), D\in \hom(W,V), E\in \hom(V,W), F\in \endo(W)$.
For any $g\in \G$ we have $g\mu =\mu (\mu^{-1} g\mu)$ with $\mu^{-1} g\mu\in \G$. In particular we have
\[ \bp \id & A(g) \\ 0&B(g) \ep \cdot \bp C&D\\ E&F \ep   = \bp C&D\\ E&F \ep \cdot  \bp \id &  A(\mu^{-1} g\mu) \\ 0&B(\mu^{-1} g\mu)\ep \]
for any $g\in \G$. Considering the first block column we see that
\[ \bp A(g)E \\ B(g)E\ep = \bp 0\\E \ep.\]
If $E\colon V\to W$ was non--trivial, then there would exist $v\in V$ such that $w=E(v)$ is non--trivial.
Given any $g\in \G$ we then have
\[ \ba{rcllll} \a^\d(g)(w)&=& \bp \id & A(g) \\ 0 &B(g) \ep \bp 0\\ w \ep &=& \bp \id & A(g) \\ 0 &B(g) \ep \bp 0\\ Ev \ep\\[4mm]
&=& \bp \id & A(g)E \\ 0 &B(g)E \ep \bp 0\\ v \ep
&=& \bp \id & 0 \\ 0 &E\ep \bp 0\\ v \ep\\[4mm]
&=&  \bp 0\\ Ev \ep
&=&w.\ea \]
  But this is not possible by the definition of $V$.
We therefore conclude that $E$ is the zero homomorphism, hence the representation $\a^\d$ restricts to a representation of $V$.
It is now straightforward to see that $\a$ preserves $W$. Since $W$ is neither zero nor all of $\KK^d$ this shows that $\a$ is reducible.
\end{proof}

\end{document}